%% file: cohom_trop.tex
\documentclass[a4paper,hidelinks,colorlinks, final]{amsart}
\pdfoutput=1

\usepackage{style}

\theoremstyle{plain}

\author{Edvard Aksnes}
\address{Department of Mathematics, University of Oslo, Oslo, Norway.}
\email{\href{edvardak@math.uio.no}{edvardak@math.uio.no}}

\author{Omid Amini}
\address{CNRS - CMLS, École polytechnique, Institut polytechnique de Paris.}
\email{\href{omid.amini@polytechnique.edu}{omid.amini@polytechnique.edu}}

\author{Matthieu Piquerez}
\address{LS2N, Inria, Nantes Université}
\email{\href{matthieu.piquerez@univ-nantes.fr}{matthieu.piquerez@univ-nantes.fr}}

\author{Kris Shaw}
\address{Department of Mathematics, University of Oslo, Oslo, Norway.}
\email{\href{krisshaw@math.uio.no}{krisshaw@math.uio.no}}

\title{Cohomologically tropical varieties}

\begin{document}

\begin{abstract}
    Given the tropicalization of a complex subvariety of the torus,
    we define a morphism between the tropical cohomology and the rational cohomology of their respective tropical compactifications. We say that the subvariety of the torus is cohomologically tropical if this map is an isomorphism for all closed strata of the tropical compactification.

    We prove that a schön subvariety of the torus is cohomologically tropical if and only if it is wunderschön and its tropicalization is a tropical homology manifold. The former property means that the open strata in
    the boundary of a tropical compactification are all connected and the mixed Hodge structures on their cohomology are pure of maximum possible weight; the latter property requires that, locally, the tropicalization verifies tropical Poincaré duality.

    We study other properties of cohomologically tropical and wunderschön varieties, and show that in a semistable degeneration to an arrangement of cohomologically tropical varieties, the Hodge numbers of the smooth fibers are captured in the tropical cohomology of the tropicalization. This extends the results of Itenberg, Katzarkov, Mikhalkin, and Zharkov.
\end{abstract}

\maketitle

\setcounter{tocdepth}{1}
\tableofcontents

\section{Overview}\label{sec:overview}

The tropicalization process transforms algebraic varieties into piecewise polyhedral objects. While losing part of the geometry, some of the invariants, such as dimension and degree, of the original variety can still be computed from its tropicalization. For the complement of a hyperplane arrangement, Zharkov shows that the tropical cohomology of the tropicalization computes the usual cohomology of the variety \cite{Zha13}. Moreover, Hacking relates the top-weight mixed Hodge structure of a variety to the homology of its tropicalization \cite{Hac08}. We are interested in determining for which varieties the tropicalization remembers the cohomology. Part of our motivation to study this question comes from the work of Deligne~\cite{Del-md} in which he gives a Hodge-theoretic characterization of maximal degenerations of complex algebraic varieties. We are more specifically interested in determining how tropicalization is related to maximal degenerations, as this question is intimately related to open problems on large scale limits of families of complex algebraic varieties, see Section~\ref{sec:discussion-max-degenerate} for a discussion.   

We introduce the relevant concepts and notation before stating our results. Let $N$ be a lattice of rank $n$, $M$ the dual lattice, and $\varT = \Spec(\C[M]) \cong (\C^*)^{n}$ the corresponding torus. We let $N_{\R}$ and $N_{\Q}$ denote $N \otimes_\Z \R$ and $N \otimes_\Z \Q$, respectively. Let $\varX \subseteq \varT$ be a non-singular subvariety of $\varT$ and denote by $\tropX = \trop(\varX)$ its tropicalization~\cite{MS15, MR18}. A unimodular fan $\Sigma$ in $N_\R$ with support $\tropX$ gives rise to a complex toric variety $\CP_\Sigma$ and a tropical toric variety $\TP_\Sigma$. Taking the closures of $\varX$ and $X$ in $\CP_\Sigma$ and $\TP_\Sigma$, respectively, gives compactifications $\varXbar$ and $\tropXbar$. We note that the compactifications depend on the choice of the fan $\Sigma$ whose support is $\trop(X)$, however, we have chosen not to indicate it in the notation for $\varXbar$ or $\tropXbar$. Here and elsewhere in the paper, we use bold letters for algebraic varieties and regular letters for tropical varieties.

For a complex variety $\varZ$, we denote by $H^\bullet(\varZ)$ the cohomology ring of $\varZ$ with coefficients in $\Q$. For a tropical variety $Z$, the $k$-th tropical cohomology group of $Z$ can be defined as $H^k(Z) \coloneqq \bigoplus_{p + q = k} H^{p,q}(Z)$, where $H^{p,q}(Z)$ is the $(p,q)$\nobreakdashes-th tropical cohomology group with $\Q$-coefficients introduced in \cite{IKMZ}, see \cref{sec:trophom}. The tropical cohomology groups together form a ring $H^{\bullet}(Z) = \bigoplus_{k} H^k(Z)$, the product structure being induced by the cup product in cohomology \cites{MZ14, GS-sheaf}. We note that the tropical cohomology of $Z$ depends only on $Z$. In particular, if $Z = \trop(\varZ)$, no information about $\varZ$ beyond $\trop(\varZ)$ goes into the recipe for computing $H^{\bullet}(\trop(\varZ))$.

The question addressed in this paper can be informally stated as follows: Under which conditions can the cohomology of $\varX$ be related to the tropical cohomology of $\trop(\varX)$?

Let $\varX$, $\Sigma$ and $\tropX$ be as above, with $\varXbar$ and $\tropXbar$ the corresponding compactifications. We define $\ctind{X,\Sigma}$ to be the ring homomorphism $\ctind{X,\Sigma} \colon H^\bullet (\tropXbar) \to H^\bullet (\varXbar)$ between the cohomologies of $\tropXbar$ and $\varXbar$, defined by composing the isomorphism $H^{k,k}(\tropXbar) \cong A^k(\CP_\Sigma)$, proved in~\cite[Theorem 1.1]{AP-FY}, with the cycle class map $A^k(\CP_\Sigma) \to H^{2k}(\CP_\Sigma)$ and the pullback morphism on cohomology associated to the embedding $\varXbar \hookrightarrow \CP_\Sigma$. The groups $H^{p,q}(\tropXbar)$ are sent to zero by $\ctind{X,\Sigma}$ for $p \neq q$. We refer to \cref{sec:cohomology_morphism} for more details.

In the following, we will use the map $\ct$ not only on $\tropXbar$ and $\varXbar$, but also on some of their subvarieties: the toric varieties $\CP_\Sigma$ and $\TP_\Sigma$ are endowed with natural stratifications induced by the cone structure of $\Sigma$. Each cone $\sigma \in \Sigma$ gives rise to the torus orbits $\varT^\sigma$ and $N^\sigma_\R$ in $\CP_\Sigma$ and $\TP_\Sigma$, respectively, with corresponding lattice $N^\sigma$. The closures in $\CP_\Sigma$ and $\TP_\Sigma$ of these orbits are denoted by $\CP_\Sigma^\sigma$ and $\TP_\Sigma^\sigma$, respectively, and are isomorphic to the complex and tropical toric varieties associated to the star fan $\Sigma^\sigma$ of $\sigma$ in $\Sigma$. Intersection with these strata induce a stratification of $\varXbar$ and $\tropXbar$. We denote by $\varX^\sigma = \varXbar \cap \varT^\sigma$ and $\tropX^\sigma=\tropXbar \cap N_\R^\sigma$ the stratum associated to $\sigma \in \Sigma$, and by $\varXbar^\sigma$ and $\tropXbar^\sigma$ their closures in $\varXbar$ and $\tropXbar$, respectively. The stratum $\varX^\sigma$ is a closed subvariety of the torus $\varT^\sigma$ and its tropicalization coincides with $\tropX^\sigma$. Moreover, the star fan $\Sigma^\sigma$ is a unimodular fan with support $\tropX^\sigma$. We thus obtain a morphism $H^\bullet (\tropXbar^\sigma) \to H^\bullet (\varXbar^\sigma)$ that we also denote by $\tau^*$.

\begin{definition}\label{def:cohom_tropical}
    Let $\varX \subseteq \varT$ be a subvariety, $\Sigma$ a unimodular fan with support $\tropX = \trop(\varX)$, and $\varXbar$ and $\tropXbar$ the corresponding compactifications. We say that $\varX$ is \emph{cohomologically tropical with respect to $\Sigma$} if the induced maps $\ctind{\tropX^\sigma,\Sigma^\sigma} \colon H^\bullet (\tropXbar^\sigma) \to H^\bullet (\varXbar^\sigma)$ are isomorphisms for all $\sigma\in \Sigma$.
\end{definition}

We show that the property of being cohomologically tropical for \emph{schön} subvarieties of tori does not depend on the chosen unimodular fan. Recall from \cite{Tev07, Hac08} that a subvariety $\varX \subseteq \varT$ is \emph{schön} if for some, equivalently for any, unimodular fan $\Sigma$ of support $\trop(\varX)$, the open strata $\varX^\sigma$, $\sigma\in\Sigma$, of the corresponding compactification are all non-singular. It also implies that the compactification $\varXbar$ is non-singular, and that $\varXbar\setminus\varX$ is a simple normal crossing divisor.

\newtheorem*{theorem:ct_fan_indep}{\cref{thm:ct_fan_independent}}
\begin{theorem:ct_fan_indep}
    Suppose that the subvariety $\varX \subseteq \varT$ is schön and let $\tropX = \trop(\varX)$ be its tropicalization. The following are equivalent.
    \begin{enumerate}
        \item There exists a unimodular fan $\Sigma$ with support $\tropX$ such that $\varX$ is cohomologically tropical with respect to $\Sigma$.
        \item For any unimodular fan $\Sigma$ with support $\tropX$, $\varX$ is cohomologically tropical with respect to $\Sigma$.
    \end{enumerate}
\end{theorem:ct_fan_indep}

Such a schön subvariety $\varX \subseteq \varT$ will be called \emph{cohomologically tropical}. For example, the linear subspaces in $\CP^n$, restricted to the torus, form a family of cohomologically tropical subvarieties. These very affine varieties are complements of hyperplane arrangements, see \cref{subsec:hyperplanecomps}. A generalization is given in~\cite{Sch21} in which Schock defines \emph{quasilinear subvarieties} of tori as those having a tropicalization which is quasilinear, in other words, $\mT$-stable in the language of \cite{AP-hodge-fan}. He shows that these subvarieties are necessarily schön. It follows from his results that quasilinear subvarieties of tori are cohomologically tropical.

\smallskip
We now introduce a class of subvarieties $\varX \subseteq \varT$ with cohomology amenable to a tropical description using the notion of mixed Hodge structures, see \cref{sec:Hodge}.

\begin{definition}\label{def:wunderschon}
    A non-singular subvariety $\varX \subseteq \varT$ of the torus is called \emph{wunderschön with respect to a unimodular fan $\Sigma$} with support $\trop(\varX)$ if all the open strata $\varX^\sigma$ of the corresponding compactification $\varXbar$ are non-singular and connected, and the mixed Hodge structure on $H^k(\varX^\sigma)$ is pure of weight $2k$ for each $k$.
\end{definition}

In particular, a point in the torus is wunderschön.
It follows from the preceding discussion that if $\varX$ is wunderschön, it is schön. Therefore, if $\varXbar$ is the compactification with respect to a unimodular fan $\Sigma$, the boundary $\varXbar \setminus \varX$ is a strict normal crossing divisor.

\smallskip
We prove that the property of being wunderschön is independent of the fan, and that the cohomology of a wunderschön variety is divisorial in the sense of \cref{sec:divisorial}.

\newtheorem*{theorem:wunderschon_fan_independent}{\cref{thm:wunderschon_fan_independent}}
\begin{theorem:wunderschon_fan_independent}
    Suppose that the subvariety $\varX \subseteq \varT$ is wunderschön with respect to some unimodular fan. Then $\varX$ is wunderschön with respect to any unimodular fan with support $\tropX =\trop(\varX)$.
\end{theorem:wunderschon_fan_independent}

\newtheorem*{theorem:wunderschon_Divisorial}{\cref{thm:wunderschon_Divisorial}}
\begin{theorem:wunderschon_Divisorial}
    Let $\varX \subseteq \varT$ be a wunderschön subvariety. Let $\varXbar$ be the compactification of\/ $\varX$ with respect to a unimodular fan $\Sigma$ with support $\tropX = \trop(\varX)$. Then the cohomology of\/ $\varXbar$ is divisorial and generated by irreducible components of $\varXbar \setminus \varX$.
\end{theorem:wunderschon_Divisorial}

A tropical variety  $\tropX$ is called a \emph{tropical homology manifold} if any open subset in $\tropX$ verifies  tropical Poincaré duality. For a tropical variety which is the support of a tropical fan, this amounts to the property that for some, equivalently for any, rational unimodular fan $\Sigma$ of support $\tropX$, the corresponding open strata $\tropX^\sigma$ verify tropical Poincaré duality for all $\sigma\in\Sigma$. In particular this implies that, for any unimodular fan $\Sigma$ of support $X$, any open subset of the corresponding tropical compactification $\tropXbar$ verifies tropical Poincaré duality.

A tropical fanfold $\tropX$ is called \emph{Kähler} if for some, equivalently for any, quasi-projective unimodular fan $\Sigma$ with support $\tropX$, and for any $\sigma \in \Sigma$, the Chow ring $A^\bullet(\Sigma^\sigma)$ verifies the Kähler package, that is, Poincaré duality, hard Lefschetz theorem and Hodge-Riemann bilinear relations. Here, for a unimodular fan $\Sigma$, the Chow ring $A^\bullet(\Sigma)$ coincides with the Chow ring of the corresponding toric variety $\CP_\Sigma$.

\smallskip
We have the following main theorem on characterization of cohomologically tropical subvarieties of tori.

\newtheorem*{theorem:main}{\cref{thm:main}}
\begin{theorem:main}
   Let\/ $\varX \subseteq \varT$ be a schön subvariety with support $\tropX = \trop(\varX)$. Then the following statements are equivalent.
    \begin{enumerate}
        \item $\varX$ is wunderschön and $\tropX$ is a tropical homology manifold,
        \item $\varX$ is cohomologically tropical.
    \end{enumerate}
    Moreover, if any of these statements holds, then $\tropX$ is Kähler.
\end{theorem:main}

We deduce the following result from the above theorem.

\newtheorem*{theorem:openCT}{\cref{thm:openCT}}
\begin{theorem:openCT}[Isomorphism of cohomology on open strata]
    Suppose that $\varX \subseteq \varT$ is schön and cohomologically tropical. Let $\Sigma$ be any unimodular fan with support $\tropX =\trop(\varX)$. Then we obtain isomorphisms
    \[ \ct\colon H^k (\tropX^{\sigma}) \simto H^k(\varX^{\sigma}) \]
    for all $\sigma \in \Sigma$ and all $k$.
\end{theorem:openCT}

\smallskip

Going beyond cohomologically tropical subvarieties of tori, and following the work of \cite{IKMZ}, one can ask the following question. Which families $\famX_t$ of complex projective varieties over the complex disk degenerating at $t=0$ have the property that the tropical cohomology of their tropical limit captures the Hodge numbers of a generic fiber in the family?

In \cref{thm:globalization}, we weaken the condition given in \cite{IKMZ} by showing that it suffices to ask the open components of the central fiber to be cohomologically tropical and schön. By \cref{thm:main}, this is equivalent to asking the maximal dimensional strata to be wunderschön and their tropicalizations to be tropical homology manifolds.

More precisely, let $\pi \colon \famX \to D^*$ be an algebraic family of non-singular algebraic subvarieties in $\CP^n$ parameterized by a punctured disk $D^*$ and with fiber $\famX_t$ over $t \in D^*$. Let $Z\subseteq \TP^n$ be the extended tropicalization of the family (see Section~\ref{sec:globalization}).

By Mumford's proof of the semistable reduction theorem~\cite{KKMS}, we find a triangulation of $Z$ (possibly after a base change) such that the extended family $\pi \colon \overline \famX \to D$ is regular and the fiber over zero $\famX_0$ is reduced and a simple normal crossing divisor. Note that since the extended family is obtained by taking the closure in a toric degeneration of $\CP^n$, each open stratum in $\famX_0$ will be naturally embedded in an algebraic torus.

\newtheorem*{theorem:globalization}{\cref{thm:globalization}}
\begin{theorem:globalization}
    Let $\pi \colon \famX \to D^*$ be an algebraic family of subvarieties in $\CP^n$ parameterized by the punctured disk and let $\pi \colon \overline \famX \to D$ be a semistable extension. If the tropicalization $Z\subseteq \TP^n$ is a tropical homology manifold and all the open strata in $\famX_0$ are wunderschön, then $H^{p,q}(Z)$ is isomorphic to the associated graded piece $W_{2p}/W_{2p-1}$ of the weight filtration in the limiting mixed Hodge structure $H_{\lim}^{p+q}$. The odd weight graded pieces in $H_{\lim}^{p+q}$ are all vanishing.

    Moreover, for $t \in D^*$, we have $\dim H^{p,q}(\famX_t) = \dim H^{p,q}(Z)$, for all non-negative integers $p$ and $q$.
\end{theorem:globalization}

Degenerations appearing in the previous theorem are necessarily maximal in the sense of~\cite{Del-md}, see Section~\ref{sec:discussion-max-degenerate} for more discussion of this connection.
\newtheorem*{theorem:max-degenerate}{\cref{thm:max-degenerate}}
\begin{theorem:max-degenerate} Notations as above, the family $\famX \to D^*$ is maximally degenerate. 
\end{theorem:max-degenerate}

We refer to the earlier work of Gross-Siebert~\cite{GS10} on integral affine manifolds with singularities and degenerations to arrangements of complete toric varieties, the work of Ruddat~\cite{Rud10} on non-necessarily maximal degenerations of Calabi-Yau varieties, and \cite{HK12, KS12, KS16, Rud20} for other interesting results connecting the topology of tropicalizations to the Hodge theory of nearby fibers.

\subsection*{Acknowledgement}

The research of E.\,A. and K.\,S. is supported by the Trond Mohn Foundation project ``Algebraic and Topological Cycles in Complex and Tropical Geometries''. O.\,A. thanks Math+, the Berlin Mathematics Research Center, for support.
M.\,P. has received funding from the European Research Council (ERC) under the European Union's Horizon 2020 research and innovation programme (grant agreement No. 101001995).

This project was started during a visit to the Norwegian Academy of Science and Letters under the Center for Advanced Study Young Fellows Project ``Real Structures in Discrete, Algebraic, Symplectic, and Tropical Geometries''. We thank the Centre of Advanced Study and Academy for their hospitality and wonderful working conditions. We thank as well the hospitality of the mathematics institute at TU Berlin where part of this research was carried out.

We are additionally grateful to an anonymous referee for providing helpful comments which helped us to improve this paper.

\section{Preliminaries}

\subsection{Subvarieties of the torus and tropicalization} \label{sec:Tropicalization}

We briefly recall the tropicalization of subvarieties of tori. Let $N$ be a lattice of rank $n$, $M$ its dual, $N_\R = N \otimes \R$, and $\varT=\varT_N =\Spec(\C[M]) \cong \torus{n}$. Let $\varX$ be a $d$-dimensional subvariety of the torus $\varT$, so that $\varX= \varV(I)$ for an ideal $I \subseteq \C[M]$. The tropicalization of $\varX$ can be described using initial ideals, see e.g. \cite[Section 3.2]{MS15},
\[ \trop(\varX) = \{ w \in N_\R \ | \ \initial_w(I) \neq \langle 1 \rangle \}. \]
A $d$-dimensional fan $\Sigma$ is \emph{weighted} if it comes equipped with a weight function $\wgt \colon \Sigma_d \to \Z$ where $\Sigma_d$ denotes the $d$-dimensional cones. A \emph{tropical fan} is a weighted fan which is pure dimensional and which satisfies the balancing condition in tropical geometry \cite[Section 3.3]{MS15}. A \emph{fanfold} is a subset of $N_\R$ which is the support of a rational fan, and it is a \emph{tropical fanfold} if it is the support of a tropical fan.

The tropicalization $X \coloneqq \trop(\varX)$ is a tropical fanfold, and any fan structure on $\trop(\varX)$ is equipped with a weight function $\wgt_{\varX}$ induced by $\varX$. If $\Sigma$ is a rational fan in $N_\R$ of support $X$ and $\eta$ is some facet of $\Sigma$, then for a generic point $w$ in the relative interior of $\eta$, the variety $\varV(\initial_w(I))$ is a union of translates of torus orbits. Then, $\wgt_{\varX}(\eta)$ is equal to the number of such torus orbit translates counted with multiplicity. This number is invariant for generic choices of points in the relative interior of $\eta$. The tropicalization endowed with the weight function $\wgt_{\varX}$ satisfies the balancing condition and thus is a tropical fanfold in $N_\R$ \cite[Section 3.4]{MS15}.

\subsection{Tropical compactifications of complex varieties}\label{sec:trop_comp}

We now briefly review the notion of \emph{tropical compactifications} introduced in \cite{Tev07}. Let $\Sigma$ be a fan in $N_\R$, and $\CP_\Sigma$ the associated toric variety. There is a bijection between cones in $\Sigma$ and torus orbits in $\CP_\Sigma$. For each $\sigma \in \Sigma$, we denote by $\varT^\sigma$ the corresponding torus orbit. The closure $\comp \varT^\sigma$ is the disjoint union $\bigsqcup_{\gamma \supseteq \sigma} \varT^\gamma$, for cones $\gamma \in \Sigma$ containing $\sigma$. For $\varX \subseteq \varT$ a subvariety and its closure $\varXbar$ in $\CP_\Sigma$, we have $\varXbar = \bigsqcup_{\sigma \in \Sigma} \varT^\sigma \cap \varXbar$. We denote the stratum $\varT^\sigma \cap \varXbar$ by $\varX^\sigma$, and its closure by $\varXbar^\sigma$. Note that, $\varXbar^\sigma = \bigsqcup_{\gamma \supseteq \sigma}\varX^\gamma$.

For $\Sigma$ a unimodular fan with support equal to $\tropX=\trop(\varX)$, the closure $\varXbar$ of $\varX$ in $\CP_\Sigma$ is compact, giving a \emph{tropical compactification} \cite[Proposition 2.3]{Tev07}. Moreover for such a $\Sigma$, the compactification $\varXbar$ of $\varX$ in $\CP_\Sigma$ is said to be \emph{schön} if the torus action $\varT \times \varXbar \to \CP_\Sigma$ is non-singular and surjective, in which case $\varXbar$ is non-singular, and the boundary $\varD\coloneqq \varXbar \setminus \varX$ is a simple normal crossing divisor \cite[Theorem 1.2]{Tev07}. The compactification $\varXbar$ is schön if and only if $\varX^\sigma$ is non-singular for each $\sigma \in \Sigma$ \cite[Lemma 2.7]{Hac08}. If $\varX$ admits a schön compactification, then any unimodular fan with support equal to $\tropX$ will provide a schön compactification \cite[Theorem 1.5]{LQ11}, and in this case, we will say that $\varX$ is \emph{schön}.

\begin{example}\label{example:schon_hypersurface}
    For $f= \sum_{I\in \Delta(f)} a_I x^I \in \C[x_1^{\pm}, \dots, x_n^{\pm}]$ a Laurent polynomial, it is pointed out in \cite[1088]{Tev07} that the very affine hypersurface $\varX=V(f)$ being schön is equivalent to the condition that $f$ is \emph{non-degenerate (with respect to its Newton Polytope)}, a concept studied in \cite{Var76a, Var76b} and \cite{Kou76}. For each face $\gamma \in \Delta(f)$ of the Newton Polytope of $f$, one defines $f_\gamma= \sum_{I\in \gamma} a_I \mathbf{x}^I$. Then $f$ is \emph{non-degenerate} if, for all $\gamma \in \Delta(f)$, the polynomials
    \[ x_1 \frac{\partial f_\gamma}{\partial x_1}, \dots, x_n \frac{\partial f_\gamma}{\partial x_n} \]
    share no common zero in $(\C^*)^n$. This implies that $\varX=V(f)$ is schön by for instance \cite[Lemma 10.3]{Var76a}.
\end{example}

\subsection{Canonical compactifications of tropical varieties}\label{subsection:canonical_comp_tropical}

Let $\Sigma$ be a rational fan in $N_{\R}$. The dimension of a cone $\sigma$ will be denoted by $\dims\sigma$, and we denote by $\Sigma_k$ the set of cones of $\Sigma$ of dimension $k$. The unique cone of dimension $0$ is denoted $\conezero$. Let $\gamma, \delta$ be two faces of $\Sigma$. We write $\gamma \subfaceeq \delta$ if $\gamma$ is a face of $\delta$. For $\delta \in \Sigma$ a cone, the saturated sublattice parallel to $\delta$ is denoted $N_\delta$, and the quotient lattice $N/{N_\delta}$ is denoted $N^\delta$, with quotient map $\pi^\delta \colon N \to N^\delta$. By a slight abuse of terminology, we also denote by $\pi^\delta$ the map of real vector spaces $N_\R \to N_\R^\delta$. Furthermore, the \emph{star} at $\delta$ is the fan $\Sigma^\delta$ in $N^\delta_\R$ whose cones are given by $\set{\pi^\delta (\sigma) \mid \delta \subfaceeq \sigma }$.

We briefly review the construction of tropical toric varieties, referring to \cite[Chapter 6.2]{MS15} for a detailed construction. Let $\T = \R \cup \{+\infty\}$ denote the tropical semi-field. Denote by $\sigma^\vee$ the semigroup of element of $M_\R$ which are nonnegative on $\sigma$. For each $\sigma \in \Sigma$, one defines $U_\sigma^\trop \coloneqq \Hom_{\text{semigroup}}(\sigma^\vee \cap M, \T)$, which can be identified with the set $\bigsqcup_{\delta \subfaceeq \sigma} N_\R^\delta$. We equip $U_\sigma^\trop$ with the subset topology of the product topology on the infinite product $\T^{\sigma^\vee \cap M}$. For $\sigma$ unimodular, $U_\sigma^\trop$ is isomorphic to $\R^{n-\dims\sigma}\times \T^{\dims\sigma}$. For $\delta \subfaceeq \sigma$, the inclusion identifies $U_\delta^\trop$ as an open subset of $U_\sigma^\trop$. The \emph{tropical toric variety} $\TP_{\Sigma}$ associated to $\Sigma$ is the space given by gluing the $U_\sigma^\trop$ along common faces, with underlying set $\bigsqcup_{\sigma \in\Sigma} N_\R^\sigma$.

Let $\Sigma$ be a fan with support $\tropX$. The \emph{canonical compactification} $\tropXbar$ of $\tropX$ relative to the fan $\Sigma$ is the closure of $\tropX$ as a subset of its tropical toric variety $\TP_\Sigma$. Furthermore, $\tropXbar$ has a cellular structure, which we denote $\overline{\Sigma}$, see Section~\ref{sec:trophom} and \cite[Section 2]{AP-FY} for details. For any cone $\sigma \in \Sigma$, we denote by $\tropX^\sigma$ the fanfold associated to $\Sigma^\sigma$. The canonical compactification ${\tropXbar^\sigma}$ of $\tropX^\sigma$ is canonically isomorphic to the closure of $\tropX^\sigma$ when considered as a subset in $N_\R^\sigma \subseteq \TP_\Sigma$, and we will denote this compactified fanfold by $\tropXbar^\sigma$, when $\Sigma$ is understood from the context. Moreover, there is an inclusion of canonical compactifications $i \colon \tropXbar^\sigma \hookrightarrow \tropXbar^\delta$ for $\delta \subfaceeq \sigma$.

When $\tropX = \trop(\varX)$ the tropical canonical compactification $\tropXbar$ relative to any fan $\Sigma$ with support $\tropX$ is the same as the extended tropicalization of the closure $\varXbar \subseteq \CP_\Sigma$ in the sense of  \cite[Section 3]{Kaj08, Pay09}.

\subsection{Mixed Hodge structures}\label{sec:Hodge}

Keeping the notation from \cref{sec:trop_comp}, let $\varX \subseteq \varT$ be a non-singular subvariety, and $\Sigma$ a unimodular fan supported on the tropicalization $\tropX = \Trop(\varX)$, so that we obtain a tropical compactification $\varXbar$ of $\varX$. Moreover suppose that the boundary $\varD \coloneqq \varXbar \setminus \varX$ is a simple normal crossing divisor. We have that $\varD = \bigcup_{\zeta \in \Sigma_1} \varXbar^\zeta$.

By \cite[Section 3]{Del-hodge2}, the \emph{logarithmic de Rham complex} $\Omega^\bullet_\varXbar(\log \varD)$ induces an isomorphism
\[ H^k(\varX;\Q)\cong \Hbb^k (\varXbar;\Omega^\bullet_\varXbar(\log \varD)), \]
for each $k$. Moreover, there is a \emph{weight filtration} $W_\bullet$ on the logarithmic de Rham complex, which gives a mixed Hodge structure on $H^k(\varX)$. This is given by the \emph{Deligne weight spectral sequence}
\[ E_{1}^{-p,q}=H^{q-2p}(\bigsqcup_{\sigma\in \Sigma_p} \varXbar^\sigma) \implies H^{q-p}(\varX), \]
which degenerates on the $E_2$-page. Below, we display the rows $E_1^{\bullet,2k+1}$ and $E_1^{\bullet, 2k}$, where the rightmost elements are in position $(0,2k+1)$ and $(0,2k)$, respectively.
\[ \begin{tikzcd}[column sep = tiny, row sep= tiny]
	{\bigoplus_{\sigma \in \Sigma_{k}} H^{1}(\varXbar^\sigma)} & {\bigoplus_{\delta \in \Sigma_{k-1}} H^{3}(\varXbar^\delta)} & \cdots & {\bigoplus_{\zeta \in \Sigma_{1}} H^{2k-1}(\varXbar^\zeta)} & { H^{2k+1}(\varXbar)} \\
	{\bigoplus_{\sigma \in \Sigma_{k}} H^{0}(\varXbar^\sigma)} & {\bigoplus_{\delta \in \Sigma_{k-1}} H^{2}(\varXbar^\delta)} & \cdots & {\bigoplus_{\zeta \in \Sigma_{1}} H^{2k-2}(\varXbar^\zeta)} & { H^{2k}(\varXbar)}
	\arrow[from=2-3, to=2-4]
	\arrow[from=2-4, to=2-5]
	\arrow[from=2-1, to=2-2]
	\arrow[from=2-2, to=2-3]
	\arrow[from=1-1, to=1-2]
	\arrow[from=1-2, to=1-3]
	\arrow[from=1-3, to=1-4]
	\arrow[from=1-4, to=1-5].
\end{tikzcd} \]

All the differentials are sums of \emph{Gysin homomorphisms} with appropriate signs. Recall that, given a unimodular fan $\Sigma$, and a pair of faces $\sigma,\delta\in \Sigma$ such that $\delta$ is a codimension one face of $\sigma$, the inclusion map $i\colon \varXbar^\sigma \to \varXbar^\delta$ induces a restriction map in cohomology $i^* \colon H^\bullet (\varXbar^\delta)\to H^\bullet (\varXbar^\sigma)$, with dual map $i_* \colon H^\bullet (\varXbar^\sigma)^*\to H^\bullet (\varXbar^\delta)^*$. Applying the Poincaré duality for both $\varXbar^\sigma$ and $\varXbar^\delta$ gives a map $\text{PD}_{\varXbar^\delta}^{-1} \circ i_* \circ \text{PD}_{\varXbar^\sigma} \colon H^{\bullet}(\varXbar^\sigma) \to H^{\bullet+2}(\varXbar^\delta)$, called the \emph{Gysin homomorphism} and denoted $\Gys_{\sigma \supface \delta}$.

Since the Deligne spectral sequence degenerates at the $E_2$ page, the cohomology of the rows $E_1^{\bullet,2k+1}$ and $E_1^{\bullet, 2k}$ yields the following associated graded elements
\begin{equation}\label{diag:E2_page_classical}
\begin{tikzcd}[column sep = tiny, row sep=tiny]
    {\Gr_{2k+1}^W (H^{k+1})} & {\Gr_{2k+1}^W(H^{k+2})} & \cdots & {\Gr_{2k+1}^W(H^{2k})} & {\Gr_{2k+1}^W(H^{2k+1})} \\
    {\Gr_{2k}^W (H^{k})} & {\Gr_{2k}^W(H^{k+1})} & \cdots & {\Gr_{2k}^W(H^{2k-1})} & {\Gr_{2k}^W(H^{2k}),}
\end{tikzcd}
\end{equation}
where $H^k \coloneqq H^k(\varX)$ and ${\Gr_{l}^W(H^{k})}$ denotes the weight $l$ part of the mixed Hodge structure on $H^k$.

Recall that a mixed Hodge structure $H$ is \emph{pure of weight $n$} if $\Gr_i^W (H)=0$ for $i\neq n$. A mixed Hodge structure $H$ is \emph{Hodge-Tate} if $\Gr_k^W(H)$ is of type $(l,l)$ if $k=2l$ and $0$ for $k$ odd, see, e.g., \cite[689]{Del-md}.

\subsection{Wunderschön varieties}

We now consider wunderschön varieties $\varX \subseteq \varT$ as introduced in \cref{def:wunderschon}. As we noted previously, wunderschön varieties are schön. In addition, we have the following.

\begin{proposition}\label{prop:wgt1wunderschoen}
    If a non-singular subvariety $\varX \subseteq \varT$ is wunderschön with respect to $\Sigma$, then the weight function of the tropicalization $\wgt_{\varX}$ is equal to one on all top dimensional faces $\eta$ of\/ $\Sigma$.
\end{proposition}

\begin{proof}
    The weight $\wgt_{\varX}(\eta) $ is equal to the intersection multiplicity of $\varXbar$ with the toric stratum $\CP_{\Sigma^\eta}$. In other words, it is the number of points in the variety $\varXbar^{\eta}$ counted with multiplicities. Since $\varX$ is wunderschön, the variety $\varXbar^{\eta}$ must consist of a single point. Hence, for all facets $\eta$ we have $\wgt_{\varX}(\eta) = 1$.
\end{proof}

A consequence of the wunderschön property is that, for each $\sigma \in \Sigma$, the even rows of the $E_2 = E_\infty$-page for $\varX^\sigma$, taking a priori the form shown in \eqref{diag:E2_page_classical}, are in fact zero except in the leftmost position, which implies that $H^k(\varX^{\sigma}) = \Gr_{2k}^W(H^{k}(\varX^\sigma))$. Moreover, the odd rows of the  $E_1$-page are all identically zero by the following lemma.

\begin{lemma}
    Let $\varX \subseteq \varT$ be a wunderschön variety with respect to $\Sigma$. Then $H^{2k-1}(\varXbar^\sigma)=0$ for $k=1,\dots,\dim (\varXbar^\sigma) $ and all $\sigma \in \Sigma$.
\end{lemma}

\begin{proof}
    The property is true for a wunderschön point. By induction on dimension, we have $H^{2k-1}(\varXbar^\sigma)=0$ for $k=1,\dots,\dim (\varXbar^\sigma)$ and all cones $\sigma$ except the central vertex $\conezero$, so that it remains to prove that $H^{2k-1}(\varXbar)=0$ for $k=1,\dots,\dim(\varXbar)$. For each such $k$, the $(2k-1)$-th row of the $E_2$-page of the Deligne spectral sequence is given by $E_2^{0,2k-1} = H^{2k-1}(\varXbar)=  \Gr_{2k-1}^W(H^{2k-1}(\varX))$ and all other terms are $0$. Since $\varX$ is wunderschön, $\Gr_{2k-1}^W(H^{2k-1}(\varX))=0$, and so $H^{2k-1}(\varXbar)=0$.
\end{proof}

Since the $E_2$-page is the cohomology of the $E_1$-page, this proves the following lemma.

\begin{lemma}\label{lemma:wunderschon_deligne_exactness}
    For $\varX \subseteq \varT$ a wunderschön variety with respect to $\Sigma$, and for each cone $\sigma \in \Sigma$ and each $k$, we have the following exact sequences
    \begin{align*}
        0 \xrightarrow{} H^k(\varX^\sigma) \xrightarrow{\res} \bigoplus_{\mu \supface \sigma \\ \dims\mu=\dims\sigma+k} H^0(\varXbar^\mu) \xrightarrow{\Gys} \bigoplus_{\nu \supface \sigma \\ \dims\nu=\dims\sigma+k-1} H^2(\varXbar^\nu) \xrightarrow{\Gys} \cdots \hspace{2cm} & \\
        \cdots \xrightarrow{\Gys} \bigoplus_{\xi \supface \sigma \\ \dims\xi=\dims\sigma+1} H^{2k-2}(\varXbar^\xi) \xrightarrow{\Gys} H^{2k}(\varXbar^\sigma) \xrightarrow{} 0,&
    \end{align*}
    where $\res$ denotes the logarithmic residue map and $\Gys$ denotes a signed sum of suitable Gysin maps.
\end{lemma}

\begin{example}[Wunderschön curves are rational] \label{ex:wunder_curves}
We classify wunderschön curves $\varX\subseteq \varT$. A tropical compactification $\varXbar$ consists of adding points to $\varX$. Points have pure mixed Hodge structure on their cohomology. Thus, for $\varX$ to be wunderschön with respect to a fan $\Sigma$, it is necessary that each stratum $X^\zeta$ for $\zeta \in \Sigma_1$ be connected, i.e., consists of a single point. The Deligne weight spectral sequence degenerates on the $E_2$ page, and is shown in \cref{fig:E_1_page_curve} and \cref{fig:E_2_page_curve}. Note that $H^2(\varX)$ is trivial. Moreover, if $\varX$ is wunderschön, then $H^1(\varXbar) = \Gr_1^W(H^1(\varX))$ must be trivial.
\begin{figure}[H]
    \centering
    \begin{minipage}[b]{.5\textwidth}
        \centering
        \begin{tikzpicture}
            \matrix (m) [matrix of math nodes,
              nodes in empty cells,nodes={minimum width=5ex,
              minimum height=5ex,outer sep=-5pt},
              column sep=2ex,row sep=0.3ex]{
                 H^{0}(\bigsqcup_{\zeta\in \Sigma_1} \varXbar^\zeta)  & H^2(\varXbar) & 2  \\
                 0 & H^1(\varXbar) & 1  \\
                 0 &  H^0(\varXbar)  & 0  \\
                -1 &  0  & \strut \quad\strut   \\};
            \draw[->, shorten <=1mm, shorten >=1mm] (m-1-1) -- (m-1-2);
            \draw[thick] ($(m-1-3.west)+(0,0.4)$) -- (m-4-3.west) ;
            \draw[thick] ($(m-4-1.north)+(-1,0)$) -- (m-4-3.north) ;
        \end{tikzpicture}

        \captionof{figure}{$E_1$-page from \cref{ex:wunder_curves}}
        \label{fig:E_1_page_curve}
    \end{minipage}%
    \begin{minipage}[b]{.5\textwidth}
        \centering
        \begin{tikzpicture}
            \matrix (m) [matrix of math nodes,
              nodes in empty cells,nodes={minimum width=5ex,
              minimum height=5ex,outer sep=-5pt},
              column sep=2ex,row sep=0.3ex]{
                \Gr_2^W(H^1(\varX))  & 0 & 2  \\
                 0 & \Gr_1^W(H^1(\varX)) & 1  \\
                 0 &  \Gr_0^W(H^0(\varX))  & 0  \\
                -1 &  0  & \strut \quad\strut   \\};
            \draw[thick] ($(m-1-3.west)+(0,0.4)$) -- (m-4-3.west) ;
            \draw[thick] ($(m-4-1.north)+(-1,0)$) -- (m-4-3.north) ;
        \end{tikzpicture}

        \captionof{figure}{$E_2$-page from \cref{ex:wunder_curves}}
        \label{fig:E_2_page_curve}
    \end{minipage}
\end{figure}

Therefore, a non-singular curve $\varX \subseteq \varT$ is wunderschön if and only if the curve $\varXbar$ is isomorphic to $\C\P^1$ and it meets each toric boundary divisor of $\CP_{\Sigma}$ in only one point. We conclude that the only wunderschön (open) curves are complements of a finite set of points in a non-singular rational curve.
\end{example}

\subsection{Tropical homology and cohomology} \label{sec:trophom}

We now briefly sketch the theory of tropical homology and cohomology, and refer to \cite{JRS18,JSS19,AP-tht,AP-FY,IKMZ,Aks21,GS-sheaf} for details. We work with $\Q$-coefficients.

Let $\Sigma$ be a rational fan in $N_\R$ with support $\tropX$. Let $\tropXbar$ be the closure of $\tropX$ inside the tropical toric variety $\TP_\Sigma$. The closure $\tropXbar$ has a cellular structure $\overline{\Sigma}$ where the cells of $\overline{\Sigma}$ consist of the closures of the cones in $\Sigma^\sigma$ for all $\sigma \in \Sigma$. In particular, each face of $\overline{\Sigma}$ is indexed by a pair of cones $\sigma, \gamma \in \Sigma$ satisfying $\gamma \supfaceeq \sigma$, and denoted $C_\gamma^\sigma \in \overline \Sigma$.  For each face $C_\gamma^\sigma \in \overline{\Sigma}$, the \emph{$p$-th multi-tangent space $\F_p(C_\gamma^\sigma)$} (with $\Q$-coefficients) is defined as
\begin{equation*}
    \F_p (C_\gamma^\sigma) \coloneqq \sum_{\eta \supfaceeq \gamma} \bigwedge^p (N_{\eta}/ N_{\sigma}) \otimes \Q  \subseteq \bigwedge^p N_{\Q}^\sigma.
\end{equation*}
Moreover, for $\alpha \subfaceeq \beta$ two faces of $\overline{\Sigma}$, there is a map $\iota_{\beta \supfaceeq \alpha} \colon \F_p(\beta) \to \F_p(\alpha)$, which is an inclusion if both faces lie in the same subfan $\Sigma^\sigma$ for some $\sigma$, or if $\alpha = C_\eta^\sigma$ and $\beta = C_\eta^{\sigma'}$ with $\eta \supfaceeq \sigma \supface \sigma'$, then $\iota_{\beta \supfaceeq \alpha}$ is induced by the projection $N^{\sigma'} \to N^\sigma$. Generally, the map $\iota_{\beta \supfaceeq \alpha}$ is defined as compositions of such inclusions and projections. Furthermore, by dualizing, we obtain the \emph{$p$-th multi-cotangent spaces $\F^p(\alpha)$} and reversed morphisms.

By selecting orientations for each of the cones $\alpha \in \overline{\Sigma}$, we obtain relative compatibility signs $\mathrm{sign}(\alpha,\beta) \in \set{\pm 1}$ for $\alpha \subface \beta$ with $\dims\beta = \dims\alpha+1$. We may thus use the multi-tangent spaces to define a chain complex
\begin{equation*}
    C_{p,q}(\overline{\Sigma}) \coloneqq \bigoplus_{\alpha \in \overline{\Sigma}_q} \F_p(\alpha),
\end{equation*}
that is, summing over faces $\alpha$ of dimension $q$ in $\overline \Sigma$, with differentials $\partial_q \coloneqq C_{p,q}(\overline{\Sigma}) \to C_{p,q-1}(\overline{\Sigma})$ defined component-wise as the maps $\mathrm{sign}(\alpha,\beta)\iota_{\beta \supface \alpha}$ when $\alpha \subface \beta$ and $\dims\beta = \dims\alpha+1$, and defined to be $0$, otherwise. Similarly, by dualizing everything, we obtain a cochain complex $C^{p,q}(\overline{\Sigma})$ for the multi-cotangent spaces.

The homology groups $H_{p,q}(\overline{\Sigma}) \coloneqq H_q(C_{p,\bullet}(\overline{\Sigma}))$ of the complex $C_{p,\bullet}(\overline{\Sigma})$ are invariants of the canonically compactified support $\tropXbar$ of the support $\tropX$ of the fan $\Sigma$. Therefore, we define the \emph{tropical homology of $\tropXbar$} as the homology $H_{p,q}(\tropXbar) \coloneqq H_q(C_{p,\bullet}(\overline{\Sigma}))$ of the complex $C_{p,\bullet}(\overline{\Sigma})$. The \emph{tropical cohomology of $\tropXbar$} is $H^{p,q}(\tropXbar)\coloneqq H^q(C^{p,\bullet}(\overline{\Sigma}))$.

In fact, tropical homology and cohomology can be defined for any rational polyhedral space. Moreover, there are various equivalent descriptions of tropical (co)homology in terms of cellular, singular, and sheaf theoretic terms \cite{IKMZ,MZ14,GS-sheaf}. For any rational polyhedral space $Z$, we set
\[ H^k(Z) \coloneqq \bigoplus_{p+q=k} H^{p,q}(Z). \]
For example, for a fanfold $X$, the tropical homology is $H_{p,q}(X)= \F_p(\conezero)$ if $q=0$ and $0$ otherwise, and the tropical cohomology of $\tropX$ is $H^{p,q}(X)= \F^p(\conezero)$ if $q=0$ and $0$ otherwise~\cite[Proposition 3.11]{JSS19}.

\smallskip
If $\tropX$ is a tropical fanfold, the balancing condition implies the existence of a \emph{fundamental class} $[\tropXbar]\in H_{d,d}(\tropXbar)$, which induces a \emph{cap product} $\frown [\tropXbar] \colon H^{p,q}(\tropXbar)\to H_{d-p,d-q}(\tropXbar)$ for each $p,q\in \set{0,\dots,d}$. When these maps are isomorphisms for all $p$ and $q$, the variety $\tropXbar$ is said to satisfy \emph{tropical Poincaré duality}.
\begin{definition}\label{def:trop_hom_mfld}
    A tropical fanfold $\tropX$ is called a \emph{tropical homology manifold} if one of the three following equivalent conditions hold:
    \begin{itemize}
        \item \textbf{There exists} a unimodular fan $\Sigma$ with support equal to $\tropX$ such that each of the canonical compactifications $\tropXbar^\sigma$ satisfies  tropical Poincaré duality, for all cones $\sigma \in \Sigma$.
        \item \textbf{For any} unimodular fan $\Sigma$ with support equal to $\tropX$, each of the canonical compactifications $\tropXbar^\sigma$ satisfies tropical Poincaré duality, for all cones $\sigma \in \Sigma$.
        \item Any open subset $U$ of $X$ satisfies tropical Poincaré duality, \ie, the tropical Poincaré duality induces an isomorphism between the tropical cohomology and the tropical Borel-Moore homology of $U$ (see \cite{JSS19,JRS18} for details). 
    \end{itemize}
A tropical variety $Z$ is called a \emph{tropical homology manifold} if any open subset $U$ of $Z$ verifies tropical Poincaré duality. \qedhere
\end{definition}

This definition corresponds to the notion of \emph{homological smoothness} in \cite{AP-homology} and to \emph{local tropical Poincaré duality spaces} in \cite{Aks21}. The equivalence of the three statements for tropical fanfolds is non-trivial and follows from Theorems 1.2 and 1.3 in \cite{AP-homology}, and Theorem 1.8 of the article \cite{AP-FY}.

\subsection{Chow rings of fans} \label{subsec:chow_ring_fan}

We now recall some facts about the Chow ring of a fan, see for instance \cites{AP-hodge-fan} for more details.

Let $\Sigma$ be a unimodular fan in a vector space $N_\R$. The Chow ring $A^\bullet(\Sigma)$ is the quotient ring
\[ A^\bullet(\Sigma) \coloneqq \rquot{\Q[\xvar_\zeta \mid \zeta \in \Sigma_1]}{(I+J)} \]
with a variable $\xvar_\zeta$ for each ray $\zeta \in \Sigma_1$. Here $I$ is the ideal generated by all monomials $\xvar_{\zeta_1}\!\cdots \xvar_{\zeta_l}$ such that the rays $\zeta_1,\dots, \zeta_l$ do not form a cone of $\Sigma$; and $J$ is the ideal generated by the expressions $\sum_{\zeta\in \Sigma_1} \langle m, \e_\zeta \rangle \xvar_\zeta$, where $\e_\zeta\in N$ is the primitive vector of the ray $\zeta$ and $m$ ranges over elements of the dual lattice $M$.

For $\sigma \in \Sigma$, we define $\xvar_\sigma \coloneqq \xvar_{\zeta_1}\!\cdots\xvar_{\zeta_k}$, where $\zeta_1, \dots, \zeta_k$ are the rays of $\sigma$. As a vector space, $A^\bullet(\Sigma)$ is generated by $\xvar_\sigma$, $\sigma\in\Sigma$. For a pair of cones $\delta \subfaceeq \sigma$, there is a Gysin map $\Gys_{\sigma \supfaceeq \delta}\colon A^\bullet(\Sigma^\sigma)\to A^{\bullet+\dims\sigma-\dims\delta} (\Sigma^\delta)$. This map is defined by mapping $\xvar_{\eta'} \in \Sigma^\sigma$ to $\xvar_{\eta'}\xvar_{\zeta_1}\!\cdots \xvar_{\zeta_r}$, where $\eta'$ is a face of $\Sigma^\sigma$, $\eta$ is the corresponding face in $\Sigma^\delta$, and $\zeta_1,\dots, \zeta_r$ are the rays of $\sigma$ not in $\delta$.

\medskip

Since $\Sigma$ is unimodular, there is an isomorphism of rings $\Phi_\Sigma \colon A^\bullet(\Sigma) \simto A^\bullet(\CP_\Sigma)$ from the Chow ring of $\Sigma$ to the Chow ring of the toric variety $\CP_\Sigma$, see e.g., \cite[Section 3.1]{Bri96}. Furthermore, the cycle class map $\mathrm{cyc}_\Sigma \colon A^\bullet(\CP_\Sigma) \to H^{2\bullet}(\CP_\Sigma)$ gives a graded ring homomorphism to cohomology, see \cite[Corollary 19.2]{Ful84}. Consider a subvariety $\varX$ of the torus, and assume that the support of $\Sigma$ is $\Trop(\varX)$. Let $\varXbar$ be the corresponding compactification. There is the restriction map of rings $r^* \colon H^{\bullet} (\CP_\Sigma) \to H^{\bullet} (\varXbar)$. Composing all these homomorphisms gives a morphism of rings $\Phi \colon A^\bullet(\Sigma) \to H^{2\bullet}(\varXbar)$ which maps $\xvar_\sigma$ to the class of $\varXbar^\sigma$.

\medskip

In the tropical world, there is a similar map. Let $\tropX$ be the support of $\Sigma$ and let $\tropXbar$ be the corresponding compactification. One can consider the composition
\[ A^\bullet(\Sigma) \to H^{2\bullet}(\T\P_\Sigma) \to H^{2\bullet}(\tropXbar) \]
mapping $\xvar_\sigma$ to the class of $\tropXbar^\sigma$. By~\cite[Theorem 1.1]{AP-FY}, this composition induces an isomorphism of rings $\bigoplus_k A^k(\Sigma) \simto \bigoplus_k H^{k,k}(\tropXbar)$. We define the inverse map $\Psi\colon H^\bullet(\tropXbar) \to A^{\bullet/2}(\Sigma)$ by mapping $(p,q)$-classes to zero if $p \neq q$. Here, by convention, $A^{k/2}(\Sigma)$ is trivial for odd $k$. If $\Sigma$ is a tropical homology manifold, $\Psi$ is an isomorphism by \cite[Theorem~1.3]{AP-FY}, that is, $H^{p,q}(\tropXbar)$ is trivial for $p\neq q$.

\subsubsection*{Kähler package}

We recall the Kähler package for Chow rings of fans, see~\cite{AP-hodge-fan}.
Assume $\Sigma$ is tropical and quasi-projective, \ie, there exists a conewise linear function $f$ on $\Sigma$ which is strictly convex in the following sense. For any $\sigma\in\Sigma$, there exists a linear map $m \in M$ such that $f-m$ is zero on $\sigma$ and strictly positive on $U\setminus\sigma$ for some open neighborhood $U$ of the relative interior of $\sigma$. To such an $f$, one can associate the element $L \coloneqq \sum_{\zeta \in \Sigma_1} f(\e_\zeta) \xvar_\zeta \in A^1(\Sigma)$. These elements coming from strictly convex functions are called \emph{ample classes}. Since $\Sigma$ is tropical, the degree map $\deg\colon A^d(\Sigma) \to \Q$ mapping $\xvar_\eta$ to $\wgt(\eta)$ for any facet $\eta$ of $\Sigma$ is a well-defined morphism.

The Chow ring $A^\bullet(\Sigma)$ is said to verify the Kähler package if the following holds:
\begin{itemize}
    \item(Poincaré duality) the pairing \[ \begin{array}{ccl}
        A^k(\Sigma) \times A^{d-k}(\Sigma) &\to& \Q, \\
        a,b\phantom{-k\!} &\mapsto& \deg(ab),
    \end{array} \]
    is perfect for any $k$;
    \item(Hard Lefschetz theorem) for any ample class $L$, the multiplication by $L^{d-2k}$ induces an isomorphism between $A^k(\Sigma)$ and $A^{d-k}(\Sigma)$ for all $k \leq d/2$;
    \item(Hodge-Riemann bilinear relations) for any $k\leq d/2$ and any ample class $L$, the bilinear map
    \[ \begin{array}{ccl}
        A^k(\Sigma) \times A^k(\Sigma) &\to& \Q, \\
        a, b &\mapsto& (-1)^k\deg(L^{d-2k}ab),
    \end{array} \]
    is positive definite on $\ker(\ \cdot L^{d-2k+1}\colon A^k(\Sigma) \to A^{d-k+1}(\Sigma))$.
\end{itemize}

A tropical fanfold $\tropX$ is called \emph{Kähler} if it is a tropical homology manifold and there exists a quasi-projective unimodular fan of support $\tropX$ such that $A^\bullet(\Sigma^\sigma)$ verifies the Kähler package for any $\sigma\in\Sigma$. In such a case, any quasi-projective unimodular fan $\Sigma$ on $\tropX$ verifies the previous property (cf. \cite{AP-hodge-fan, AP-homology}).

\subsection{Tropical Deligne resolution}\label{subsec:TropicalDeligne}

Let $\Sigma$ be a unimodular fan on some tropical homology manifold $X$. Let $\delta \subfaceeq \sigma$ be two faces of $\Sigma$. The inclusion $i^\trop \colon \tropXbar^\sigma \to \tropXbar^\delta$ of canonically compactified fanfolds, both satisfying tropical Poincaré duality, gives a homomorphism $i_{*}^\trop \colon H_k(\tropXbar^\sigma) \to H_k(\tropXbar^\delta)$. Applying the tropical Poincaré duality for both $\tropXbar^\sigma$ and $\tropXbar^\delta$, this gives a map $\text{PD}_{\tropXbar^\delta}^{-1} \circ i_*^{\trop} \circ \text{PD}_{\tropXbar^\sigma} \colon H^{k}(\tropXbar^\sigma) \to H^{k+2(\dims{\sigma} - \dims{\delta})}(\tropXbar^\delta)$, called the \emph{tropical Gysin homomorphism} and denoted $\Gys_{\sigma\supfaceeq \delta}^\trop$.

In \cite[Theorem 1.6]{AP-homology}, it is shown that for a fanfold $\tropX$ which is a tropical homology manifold and a simplicial fan $\Sigma$ with support $\tropX$, there are \emph{tropical Deligne resolutions}, i.e., exact sequences for any $k$,
\begin{align*}
    0 \longrightarrow {H^k(\tropX)} \longrightarrow {\bigoplus_{\sigma \in \Sigma_{k}} H^0(\tropXbar^\sigma)} \longrightarrow {\bigoplus_{\delta \in \Sigma_{k-1}} H^2(\tropXbar^\delta)} \longrightarrow \cdots \hspace{2cm}& \\
    \cdots\longrightarrow {\bigoplus_{\zeta \in \Sigma_{1}} H^{2k-2}(\tropXbar^\zeta)} \longrightarrow {H^{2k}(\tropXbar)} \longrightarrow {0,} &
\end{align*}
where the first non-zero morphism is given by integration (that is, by the evaluation of the element $\alpha \in H^k(\tropX)$ at the canonical multivector of each face $\sigma \in \Sigma_k$), and all subsequent maps are given by the tropical Gysin homomorphisms (with appropriate signs \cite[Section 3]{AP-homology}).

\section{The induced morphism on cohomology by tropicalization} \label{sec:cohomology_morphism}

The aim of this section is to define a map relating tropical cohomology to classical cohomology, as well as to prove \cref{thm:compare_gysin}, which relates Gysin maps in tropical and classical cohomology.

\begin{definition}\label{def:ct_map}
    Let $\varX \subseteq \varT$ be a subvariety and $\Sigma$ a unimodular fan with support $\tropX=\trop(\varX)$, and $\varXbar$ and $\tropXbar$ be the compactifications of\/ $\varX$ and $\tropX$ with respect to~$\Sigma$. We define
    \[ \ctind{X,\Sigma} \colon H^\bullet (\tropXbar) \to H^\bullet (\varXbar) \]
    to be the ring homomorphism defined as the composition of the maps $\Psi\colon H^\bullet (\tropXbar) \to A^{\bullet/2}(\Sigma)$ with $ \Phi \colon A^{\bullet/2} (\tropXbar) \to H^\bullet (\varXbar)$ from \cref{subsec:chow_ring_fan}.
\end{definition}

The map $\ct$ is the morphism comparing the tropical and classical cohomology in order to define \emph{cohomologically tropical} varieties in \cref{def:cohom_tropical}.

We will now relate the classical and tropical Gysin maps through the map $\ct$. This will be useful later for comparing Deligne sequences.

\begin{proposition}\label{thm:compare_gysin}
    Let $\tropX = \Trop(\varX)$ the be tropicalization of a subvariety $\varX \subseteq \varT$, $\Sigma$ a unimodular fan with support $\tropX$, with $\sigma,\delta \in \Sigma$ such that $\delta$ is a face of $\sigma$ of codimension one, giving inclusion maps $\varXbar^\sigma \to \varXbar^\delta$ and $\tropXbar^\sigma \to \tropXbar^\delta$. Then the following diagram commutes:
    \[\begin{tikzcd}
        H^{k}(\tropXbar^\sigma) & H^{k}(\varXbar^\sigma) \\
        H^{k+2}(\tropXbar^\delta) & H^{k+2}(\varXbar^\delta).
        \arrow["{\Gys_{\sigma \supface \delta}^\trop}", from=1-1, to=2-1]
        \arrow["{\ctind{\varX,\Sigma}}"', from=1-1, to=1-2]
        \arrow["{\Gys_{\sigma \supface \delta}}", from=1-2, to=2-2]
        \arrow["{\ctind{\varX',\Sigma'}}", from=2-1, to=2-2]
    \end{tikzcd}\]
\end{proposition}

\begin{proof}
    Expanding the definition of $\ct$, we obtain the following diagram
    \[\begin{tikzcd}
        {H^\bullet (\tropXbar^\sigma)} & {A^{\bullet/2} (\Sigma^\sigma)} & {H^\bullet (\varXbar^\sigma)} \\
        {H^{\bullet+2} (\tropXbar^\delta)} & {A^{\bullet/2+1}(\Sigma^\delta)} & {H^{\bullet+2} (\varXbar^\delta).}
        \arrow["{\Gys_{\sigma\supface \delta}^\trop}"', from=1-1, to=2-1]
        \arrow["{\Gys_{\sigma\supface \delta}}", from=1-2, to=2-2]
        \arrow["{\Phi}", from=1-2, to=1-3]
        \arrow["{\Phi}", from=2-2, to=2-3]
        \arrow["{\Psi}", from=1-1, to=1-2]
        \arrow["{\Psi}", from=2-1, to=2-2]
        \arrow["{\Gys_{\sigma\supface \delta}}", from=1-3, to=2-3]
    \end{tikzcd}\]
    The first square is commutative by \cite[Remark 3.15]{AP-tht}, in light of \cite[Theorem 1.1]{AP-FY}.
    The commutativity of the second square follows from the functoriality of the cycle class map in light of \cite[Section 3.2]{Bri96} and \cite[Section 19.2]{Ful84}.
\end{proof}

\begin{remark}\label{remark:diagram}
    Let $\varX \subseteq \varT_N$ and $\varX'\subseteq \varT_{N'}$ be two non-singular subvarieties of tori associated to two lattices $N$ and $N'$, with $\tropX$ and $\tropX'$ the corresponding tropicalizations, and two unimodular fans $\Sigma$ and $\Sigma'$ with supports $\tropX$ and $\tropX'$, respectively.

    Assume there exists a morphism of lattices $\phi \colon N \to N'$ which takes cones of $\Sigma$ to cones of $\Sigma'$ such that the induced map $\phi\rest{\tropX}\colon \tropX \to \tropX'$ is surjective. This makes the induced morphism of toric varieties $f \colon \CP_\Sigma \to \CP_{\Sigma'}$ proper \cite[Section 2.4]{Ful-toric}. We denote by $f^\trop\colon \TP_\Sigma \to \TP_{\Sigma'}$ the induced morphism on tropical toric varieties.

    Furthermore, suppose that $f(\varX)=\varX'$. Since $\varXbar$ is compact we have that $f(\varXbar)=\shiftcomp[3]{f(\varX)}=\varXbar'$. This also gives $f^\trop (\tropXbar)=\tropXbar'$ for the canonical compactifications of $\tropX$ and $\tropX'$ with respect to $\Sigma$ and $\Sigma'$. One can then prove the commutativity of the following diagram
    \[ \begin{tikzcd}[ampersand replacement=\&]
        {H^\bullet (\tropXbar')} \rar{\ctind{\varX,\Sigma}} \dar["f^{\trop,*}"']{}\& {H^\bullet (\varXbar')} \dar{f^*} \\
        {H^\bullet (\tropXbar)} \rar["\ctind{\varX',\Sigma'}"']{} \& {H^\bullet (\varXbar).}
      \end{tikzcd}\vspace{-1.5em} \]\qedhere
\end{remark}

\begin{proposition}\label{prop:injectivity}
    Let $\varX \subseteq \varT$ be a subvariety of complex dimension $d$ and $\Sigma$ a unimodular fan with support $\tropX=\trop(\varX)$, and\/ $\varXbar$ and $\tropXbar$ be the compactifications of\/ $\varX$ and $\tropX$ with respect to $\Sigma$. Suppose $\tropXbar$ satisfies tropical Poincaré duality and\/ $\varXbar$ is non-singular. Then $ \ctind{X,\Sigma} \colon H^\bullet (\tropXbar) \to H^\bullet (\varXbar)$ is injective.
\end{proposition}

\begin{proof}
    Both maps $\Psi\colon H^{2d}(\tropXbar) \to A^d(\Sigma)$ and $\Phi \colon A^d(\Sigma) \to H^{2d}(\varXbar)$ commute with the corresponding degree maps. Now for both tropical and classical cohomology, the fact that the products induce perfect pairings implies that $\ctind{X,\Sigma}$ is injective.
\end{proof}

\section{Irrelevance of fan} \label{sec:fanirrelevance}

To be schön, wunderschön, cohomologically tropical, Kähler, or a tropical homology manifold are all properties of the form \enquote{there exists a fan $\Sigma$ such that a specific property holds} with some restriction on the fan, as unimodularity for instance. Informally, we say that such a property is \emph{fan irrelevant} if we can replace \enquote{there exists a unimodular fan} by \enquote{for any unimodular fan} (this is strongly linked with the notion of $\mT$-stability for properties of tropical fans in \cite{AP-hodge-fan}). It is already known that to be schön, Kähler or a tropical homology manifold is fan irrelevant. In this section we prove \cref{thm:ct_fan_independent,thm:wunderschon_fan_independent} about the fan irrelevance of being cohomologically tropical and wunderschön. We begin with a lemma.

\begin{lemma} \label{lem:cohom_trop_trop_hom_mfld}
    Suppose a schön subvariety $\varX \subseteq \varT$ is cohomologically tropical. Then, the tropicalization $\tropX =\trop(\varX)$ is a tropical homology manifold.
\end{lemma}

\begin{proof}
    Let $\Sigma$ be a unimodular fan whose support is $\trop(\varX)$. It follows that the cohomology groups $H^\bullet(\tropXbar^\sigma)$ are all isomorphic to the cohomology groups $H^\bullet(\varXbar^\sigma)$, and so they verify Poincaré duality. We infer that $\tropX$ is a tropical homology manifold.
\end{proof}

Let $\varX$ be a schön subvariety of the torus which is cohomologically tropical. It follows from the previous lemma and the fan irrelevance of being a tropical homology manifold that all the cohomology groups $H^{p,q}(\tropXbar)$ are vanishing provided that $p\neq q$, for the canonical compactification $\tropXbar$ of $\tropX$ with respect to any unimodular fan with support $\tropX$.

Let $\Sigma$ be a unimodular fan with support the fanfold $\tropX$, and let $\sigma$ be a cone in $\Sigma$ of dimension at least two. Let $\Sigma'$ be the barycentric star subdivision of $\Sigma$ obtained by star subdividing $\sigma$, see e.g.~\cite{AP-hodge-fan, Wlo03}. Denote by $\rho$ the new ray in $\Sigma'$. Let $\tropXbar$ and $\tropXbar'$ be the compactifications of $\tropX$ with respect to $\Sigma$ and $\Sigma'$, respectively.

The following theorem provides a description of the Chow ring of $\Sigma'$ in terms of the Chow rings of $\Sigma$ and $\Sigma^\sigma$.

\begin{theorem}[Keel's lemma] \label{thm:keel}
    Let $\J$ be the kernel of the surjective map $\ii^*_{\conezero\subfaceeq\sigma}\colon A^\bullet(\Sigma)\to A^\bullet(\Sigma^\sigma)$ and let
    \[ P(\Tvar) \coloneqq \prod_{\zeta\subface \sigma \\ \dims\zeta=1}(\xvar_\zeta+\Tvar). \]

    There is an isomorphism of Chow groups given by the map
    \[ \chi\colon \rquot{A^\bullet(\Sigma)[\Tvar]}{(\J \Tvar +P(\Tvar))} \simto A^\bullet(\Sigma') \]
    which sends $\Tvar$ to $-\xvar_\rho$ and which verifies
    \[ \forall \zeta \in \Sigma_1, \qquad
    \chi(\xvar_\zeta) = \begin{cases}
        \xvar_\zeta+\xvar_\rho & \textrm{if $\zeta \subface \sigma$,}\\
        \xvar_\zeta & \textrm{otherwise}.
    \end{cases} \]
    In particular this gives a vector space decomposition of $A^\bullet(\Sigma')$ as
    \begin{equation} \label{eqn:keel_chow}
        A^\bullet(\Sigma')\cong A^\bullet(\Sigma)\oplus A^{\bullet-1}(\Sigma^\sigma)\Tvar \oplus \dots \oplus A^{\bullet-\dims{\sigma}+1}(\Sigma^\sigma)\Tvar^{\dims{\sigma}-1}.
    \end{equation}

    In addition, if $\tropX$ is the tropicalization of a schön subvariety $\varX \subseteq \varT$, and $\varXbar$ and $\varXbar'$ are compactifications of\/ $\varX$ with respect to $\Sigma$ and $\Sigma'$, respectively, then we have an isomorphism
    \[ H^\bullet(\varXbar')\cong \rquot{H^\bullet(\varXbar)[\Tvar]}{(\J \Tvar +P(\Tvar))}, \]
    and the decomposition
    \begin{equation} \label{eqn:keel_cohomology}
        H^\bullet(\varXbar')\cong H^\bullet(\varXbar)\oplus H^{\bullet-1}(\varXbar^\sigma)\Tvar \oplus \dots \oplus H^{\bullet-\dims{\sigma}+1}(\varXbar^\sigma)\Tvar^{\dims{\sigma}-1}.
    \end{equation}
    Here, by an abuse of notation, the variable $\Tvar$ denotes the image of $-\xvar_\rho$ in $H^2(\varXbar')$ for the induced map $A^\bullet(\Sigma') \to H^\bullet(\varXbar')$, $\J$ is the kernel of $H^\bullet(\varXbar)\to H^\bullet(\varXbar^\sigma)$, and $P(\Tvar)$ is the image of\/ $\prod_{\zeta\subface \sigma \\ \dims\zeta=1}(\xvar_\zeta+\Tvar)$ in $H^\bullet(\varXbar)[\Tvar]$ under the map $A^\bullet(\Sigma) \to H^{\bullet}(\varXbar)$.
\end{theorem}

Decomposition \eqref{eqn:keel_chow}, for instance, means that for any $1 \leq k \leq \dims\sigma$, we have a natural injective map
\[ A^\bullet(\Sigma^\sigma) \hookrightarrow A^\bullet(\Sigma'^{\rho}) \xhookrightarrow{-\Gys_{\rho\supface\conezero}\ } A^{\bullet+1}(\Sigma') \xhookrightarrow{\Tvar^{k-1}} A^{\bullet+k}(\Sigma'). \]
The piece $A^\bullet(\Sigma^\sigma)\Tvar^k$ in the above decomposition then denotes the image of the above map. We refer to~\cite{Kee92} and~\cite{AP-hodge-fan} for more details and the proof.

\smallskip
Two unimodular fans with the same support are called \emph{elementary equivalent} if one can be obtained from the other by a barycentric star subdivision. The \emph{weak equivalence} between unimodular fans with the same support is then defined as the transitive closure of the elementary equivalence relation. We will need the weak factorization theorem, stated as follows.

\begin{theorem}[Weak factorization theorem~\cite{Wlo97, Mor96}] \label{thm:equivalent_fan}
    Two unimodular fans with the same support are always weakly equivalent.
\end{theorem}

We are now in a position to prove the independence of being cohomologically tropical from the chosen fan for schön varieties.

\begin{theorem} \label{thm:ct_fan_independent}
    Suppose that the subvariety $\varX \subseteq \varT$ is schön and let $\tropX = \trop(\varX)$ be its tropicalization. The following are equivalent.
    \begin{enumerate}
        \item There exists a unimodular fan $\Sigma$ with support $\tropX$ such that $\varX$ is cohomologically tropical with respect to $\Sigma$.
        \item For any unimodular fan $\Sigma$ with support $\tropX$, $\varX$ is cohomologically tropical with respect to $\Sigma$.
    \end{enumerate}
\end{theorem}

\begin{proof}
    Suppose that the subvariety $\varX$ of the torus $\varT$ is schön. Let $\tropX= \trop(\varX)$. Let $\Sigma$ be a unimodular fan with support $\tropX$ such that $\varX$ is cohomologically tropical with respect to $\Sigma$. Let $\Sigma'$ be a second unimodular fan with support $\tropX$. We need to prove that $\varX$ is cohomologically tropical with respect to $\Sigma'$. By the weak factorization theorem, it will be enough to assume that $\Sigma$ and $\Sigma'$ are elementary equivalent.

    We consider the compactifications $\varXbar'$ and $\tropXbar'$ of $\varX$ and $\tropX$ with respect to $\Sigma'$, and those with respect to $\Sigma$ by  $\varXbar$ and $\tropXbar$.

    \smallskip
    Consider first the case where $\Sigma'$ is obtained as a barycentric star subdivision of $\Sigma$. Denote by $\sigma$ the cone of $\Sigma$ which has been subdivided and by $\rho$ the new ray of $\Sigma'$.

    We start by explaining the proof of the isomorphism ${H^\bullet (\tropXbar')} \simto {H^\bullet (\varXbar')}$. We use the notation preceding \cref{thm:equivalent_fan}. By Keel's lemma, we get
    \[ A^\bullet(\Sigma')\cong \rquot{A^\bullet(\Sigma)[\Tvar]}{(\J \Tvar+P(\Tvar))} \qquad \textrm{and} \qquad H^\bullet(\varXbar')\cong \rquot{H^\bullet(\varXbar)[\Tvar]}{(\J \Tvar+P(\Tvar))} \]
    with $\J$ and $P(\Tvar)$ as in \cref{thm:keel}.

    By~\cite[Theorem 1.1]{AP-FY}, see \cref{subsec:chow_ring_fan}, we have isomorphisms $A^p(\Sigma')\simto H^{p,p}(\tropXbar')$ and $A^p(\Sigma)\simto H^{p,p}(\tropXbar)$ for each $p$. Moreover, since $\varX$ is cohomologically tropical by \cref{lem:cohom_trop_trop_hom_mfld}, all the cohomology groups $H^{p,q}(\tropXbar')$ and $H^{p,q}(\tropXbar)$ are vanishing for $p\neq q$.

    The isomorphism ${H^\bullet (\tropXbar')} \simto {H^\bullet (\varXbar')}$ now follows from the commutativity of the diagram in \cref{remark:diagram}, the isomorphisms ${H^\bullet (\tropXbar)} \simto {H^\bullet (\varXbar)}$ and ${H^\bullet (\tropXbar^\sigma)} \simto {H^\bullet (\varXbar^\sigma)}$, and the compatibility of the decompositions in Keel's lemma in the tropical and algebraic settings with respect to these isomorphisms.

    \smallskip
    Consider now an arbitrary cone $\delta$ of $\Sigma'$ and denote by $\eta$ the smallest cone of $\Sigma$ which contains $\delta$. The star fan $\Sigma'^\delta$ of $\delta$ in $\Sigma'$ is isomorphic to a product of two fans $\Delta \times \Theta$ with $\Delta$ a unimodular fan living in $N^\eta_\R$ and $\Theta$ a unimodular fan living in $\rquot{N_{\sigma, \R}}{N_{\delta\cap \sigma,\R}}$. In the case $\eta \not\subfaceeq \sigma$, the first fan $\Delta$ coincides with the star fan $\Sigma^\eta$ of $\eta$ in $\Sigma$. Otherwise, when $\eta \subfaceeq \sigma$, $\Delta$ is the fan obtained from $\Sigma^\eta$ by subdividing the cone $\sigma/\eta$. The other fan $\Theta$ is $\conezero$ unless $\delta$ contains the ray $\rho$ in which case, $\Theta$ is the fan of the projective space of dimension $\dims \sigma - \dims{\sigma \cap \delta}$. Similarly, $\varXbar'^\delta$ admits a decomposition into a product $\varY \times \varZ$, where $\varY = \varXbar^\eta$ in the case $\eta\not\subfaceeq \sigma$, and $\varY$ is the blow-up of $\varXbar^\sigma$ in $\varXbar^\eta$ in the other case $\eta \subfaceeq \sigma$. And $\varZ$ is $\CP^0$, that is a point, unless $\delta$ contains $\rho$ in which case $\varZ\cong \CP^{\dims \sigma - \dims{\sigma \cap \delta}}$.

    The isomorphism ${H^\bullet (\tropXbar'^\delta)} \simto {H^\bullet (\varXbar'^\delta)}$ for $\delta$ can be then obtained from the above description, and by observing that when $\sigma$ is face of $\eta$ and $\Delta$ is the subdivision of $\eta/\sigma$ in $\Sigma^\eta$, we can apply the argument used in the first treated case above to $\varXbar^\eta$ and $\tropXbar^\eta$ to conclude.

    \smallskip
    Consider now the case where $\Sigma$ is obtained as a barycentric star subdivision of $\Sigma'$. We only discuss the isomorphism ${H^\bullet (\tropXbar')} \simto {H^\bullet (\varXbar')}$, the other isomorphisms ${H^\bullet (\tropXbar'^\delta)} \simto {H^\bullet (\varXbar'^\delta)}$ for $\delta\in \Sigma'$ can be obtained by using the preceding discussion. The cohomology of $\tropXbar'$ appears as a summand of the cohomology of $\tropXbar$ according to the decomposition in Keel's lemma. Similarly, the cohomology of $\varXbar'$ is a summand of the cohomology of~$\varXbar$. Using the compatibility of the decompositions in the Keel's lemma, the isomorphism ${H^\bullet (\tropXbar)} \simto {H^\bullet (\varXbar)}$ induces an isomorphism ${H^\bullet (\tropXbar')} \simto {H^\bullet (\varXbar')}$ between the two summands.
\end{proof}

\begin{theorem}\label{thm:wunderschon_fan_independent}
    Suppose that the subvariety $\varX \subseteq \varT$ is wunderschön with respect to some unimodular fan. Then $\varX$ is wunderschön with respect to any unimodular fan with support $\tropX =\trop(\varX)$.
\end{theorem}

\begin{proof}
    The proof of this theorem is similar to the proof given above for \cref{thm:ct_fan_independent}. We omit the details.
\end{proof}

\section{Divisorial cohomology}\label{sec:divisorial}

In this section, we prove \cref{thm:wunderschon_Divisorial} which states that the cohomology of a wunderschön variety is divisorial.

The cohomology of a non-singular algebraic variety $\varZ$ is \emph{divisorial} if there is a surjective ring homomorphism $\Q[ \xvar_1, \dots, \xvar_s ] \to H^{\bullet}(\varZ)$ such that the image of each $\xvar_i$ is $[\varD_i] \in H^2(\varZ)$, the Poincaré dual of some divisor $\varD_i$ of $\varZ$. Similarly, the Chow ring $A^{\bullet}(\varZ)$ is \emph{divisorial} if there is a surjective ring homomorphism $\Q[ \xvar_1, \dots, \xvar_s ] \to A^{\bullet}(\varZ)$ such that the image of each $\xvar_i$ is the class of a divisor $\varD_i$ of $\varZ$. In this case, we also say that the (Chow) cohomology of $\varZ$ is generated by the divisors $\varD_1, \dots, \varD_s$. Notice that if $\varZ$ is projective and its cohomology is divisorial, then all its cohomology is generated by algebraic cycles and the Hodge structure on the cohomology is Hodge-Tate.

The Chow ring of any non-singular complex toric variety is divisorial and generated by the toric boundary divisors, see \cite[Section 3.1]{Bri96} and \cref{subsec:chow_ring_fan}. It follows, using our previous notations, that if the the map $\ctind{\varX,\Sigma} \colon H^\bullet (\tropXbar) \to H^\bullet (\varXbar)$ is a surjection, then the cohomology of\/ $\varXbar$ is divisorial and generated by the irreducible components of\/ $\varXbar \setminus \varX$.

\begin{theorem}\label{thm:wunderschon_Divisorial}
    Let $\varX \subseteq \varT$ be a wunderschön subvariety. Let $\varXbar$ be the compactification of\/ $\varX$ with respect to a unimodular fan $\Sigma$ with support $\tropX = \trop(\varX)$. Then the cohomology of\/ $\varXbar$ is divisorial and generated by irreducible components of $\varXbar \setminus \varX$.
\end{theorem}

\begin{proof}
    We proceed by induction on the dimension of $\varX$. If $\varX$ is a point, then this is trivial. Notice also that if $\varX$ is a wunderschön curve then $\varXbar$ must be $\CP^1$ and hence the cohomology is divisorial as $H^{\bullet}(\CP^1)\cong \Q[\xvar]/\langle \xvar^2 \rangle$.

    We have the following commutative diagram
    \[ \begin{tikzcd}
        \bigoplus_{\rho\in \Sigma_1} \Q[\xvar_{\zeta}\ |\ \zeta\in\Sigma_1\text{ and }(\rho+\zeta)\in\Sigma_2] \rar{\bigoplus_\rho f_\rho}\dar{\bigoplus_\rho -\cdot\xvar_\rho} & \bigoplus_{\rho\in \Sigma_1} H^\bullet(\varXbar^\rho) \dar{\Gys} \\
        \Q[\xvar_\zeta\ |\ \zeta\in\Sigma_1] \rar{f} & H^{\bullet+2}(\varXbar),
    \end{tikzcd} \]
    where $\rho+\zeta$ is the cone generated by the rays $\rho$ and $\zeta$, the $f_\rho$ are surjective ring homomorphisms which send $\xvar_\zeta$ to $[\varXbar^{\rho+\zeta}]$, and $f$ maps $\xvar_\zeta$ to $[\varXbar^\zeta]$. Since $\varX$ is wunderschön the maps
    \[ \bigoplus_{\rho \in \Sigma_1} H^{k}(\varXbar^\rho) \to H^{k +2}(\varXbar) \]
    from the Deligne weight spectral sequence are all surjections for $k \geq 0$ and we deduce that $f$ is surjective. Therefore, the cohomology of $\varXbar$ is divisorial and is generated by the components of $\varXbar \setminus \varX$.
\end{proof}

\section{Proof of the main theorem}

We now turn to proving \cref{thm:main}.
\begin{theorem}\label{thm:main}
    Let\/ $\varX \subseteq \varT$ be a schön subvariety with support $\tropX =\trop(\varX)$. Then the following statements are equivalent.
    \begin{enumerate}
        \item\label{thm:main:wunder} $\varX$ is wunderschön and $\tropX$ is a tropical homology manifold,
        \item\label{thm:main:ct} $\varX$ is cohomologically tropical.
    \end{enumerate}
    Moreover, if any of these statements holds, then $\tropX$ is Kähler.
\end{theorem}

\begin{proof}
    We begin by assuming that $\varX$ is wunderschön and that $\tropX$ is a tropical homology manifold, and prove that $\varX$ is cohomologically tropical. We must show that the maps $\ctind{\varX^\sigma,\Sigma^\sigma} \colon H^\bullet (\tropXbar^\sigma) \to H^\bullet (\varXbar^\sigma)$ are isomorphisms for all $\sigma\in \Sigma$. Notice that $\varX$ is non-singular since it is wunderschön.

    If $\varX$ is of dimension $0$ and wunderschön it consists of a single point. Therefore, its tropicalization is a point of weight $1$ thus $\varX$ is cohomologically tropical. We proceed by induction on the dimension of $\varX$. Therefore, we can assume that each of the $\varX^\sigma$ is cohomologically tropical for all cones $\sigma \in \Sigma$ not equal to the origin.

    Since $\varX$ is schön, let $\varD = \varXbar \setminus \varX$ be the simple normal crossing divisor of the compactification. The Deligne weight spectral sequence for the tropical compactification $(\varXbar, \varD)$ of $\varX$ abuts in the associated graded objects of the weight filtration of the cohomology of $H^k(\varX)$. Since $\varX$ is wunderschön, the $E_1$-page of Deligne spectral sequence extends to exact rows by \cref{lemma:wunderschon_deligne_exactness}, with the morphisms being sums of Gysin maps. In the tropical setting, since $X$ is a tropical homology manifold, there are tropical Deligne resolutions \cref{subsec:TropicalDeligne}, where the maps are sums of tropical Gysin maps.

    Now by induction, $\ctind{\varX^\sigma, \Sigma^\sigma}\colon H^\bullet (\tropXbar^\sigma) \to H^\bullet(\varXbar^\sigma)$ is an isomorphism, and moreover the appropriate commutative diagrams using the classical and tropical Gysin maps commute by \cref{thm:compare_gysin}. We may therefore identify the two exact sequences. Applying the five lemma in the cases $k\geq 2$, exactness gives us isomorphisms $H^k(\tropX)\to H^k(\varX)$ and $\ctind{X,\Sigma}\colon H^{2k}(\tropXbar) \to H^{2k}(\varXbar)$. For $k=0$, since $\varX$ is assumed to be connected, there is an isomorphism $H^0(\tropXbar)\cong \Q \cong H^0(\varXbar)$, and it merely remains to show the claim for $k=1$.

    We consider the following commutative diagram
    \[ \begin{tikzcd}
        0 & {H^1(\tropX)} & {\bigoplus\limits_{\zeta \in \Sigma_{1}} H^0(\tropXbar^\zeta)} & {H^{2}(\tropXbar)} & 0 \\
        0 & {H^1(\varX)} & {\bigoplus\limits_{\zeta \in \Sigma_{1}} H^0(\varXbar^\zeta)} & {H^2(\varXbar)} & 0
        \arrow[from=2-2, to=2-3]
        \arrow["g", from=2-3, to=2-4]
        \arrow[from=2-4, to=2-5]
        \arrow[from=2-1, to=2-2]
        \arrow[from=1-4, to=1-5]
        \arrow[from=1-1, to=1-2]
        \arrow[from=1-2, to=1-3]
        \arrow[from=1-3, to=1-4]
        \arrow["{\ctind{\varX,\Sigma}}", from=1-4, to=2-4]
        \arrow["{\bigoplus \ctind{\varX^\sigma, \Sigma^\sigma}}", from=1-3, to=2-3]
        \arrow[from=1-2, to=2-2].
    \end{tikzcd} \]

    By induction, the middle vertical arrow is an isomorphism, and we wish to show that the rightmost vertical arrow is an isomorphism. By a diagram chase, exactness of the lower row implies that this arrow is surjective. The injectivity follows from \cref{prop:injectivity}. Therefore, the map $\ctind{\varX^\sigma, \Sigma^\sigma} \colon {H^{2}(\tropXbar)} \to {H^2(\varXbar)} $ is an isomorphism. Together with our induction assumption on the maps $\ctind{\varX^\sigma, \Sigma^\sigma}$ this proves that $\varX$ is cohomologically tropical.

    \smallskip
    Now assume that $\varX$ is cohomologically tropical. By \cref{lem:cohom_trop_trop_hom_mfld}, we know that $\tropX$ is a tropical homology manifold. It remains to show that $\varX$ is wunderschön. We again proceed by induction on dimension as the case for points is trivial. We equip $\varX$ with the tropical compactification $\varXbar$ given by $\Sigma$, such that all open $\varX^\sigma$ are wunderschön by induction, for $\sigma$ different from the central vertex of $\Sigma$. We have $H^0(\varXbar) \cong H^0(\tropXbar)$ by hypothesis, and $H^0(\tropXbar)\cong \Q$, thus $\varXbar$ is connected and so is~ $\varX$. It remains to show that the mixed Hodge structure on $H^k(\varX)$ is pure of weight $2k$ for each $k$. This follows from comparing the Deligne weight spectral sequence and tropical Deligne resolution by \cref{thm:compare_gysin}, using that all the maps $\ctind{\varX^\sigma, \Sigma^\sigma}$ are isomorphisms. Hence $\varXbar$ is wunderschön.

    \smallskip
    Finally we prove that if $\varX$ is cohomologically tropical, then $\tropX$ is Kähler. By  \cref{lem:cohom_trop_trop_hom_mfld}, we know that $\tropX$ is a tropical homology manifold. There exists a unimodular fan $\Sigma$ with support $\tropX$ such that $\Sigma$ is quasi-projective. It follows that the Chow rings $A^{\bullet/2}(\Sigma^\sigma)$, $\sigma \in \Sigma$, are isomorphic to $H^\bullet(\varXbar^\sigma)$. Moreover, since $\Sigma$ is quasi-projective, and $\varX$ is schön, $\varXbar^\sigma$ is a non-singular projective variety, and so its cohomology verifies the Kähler package. We conclude that $\tropX$ is Kähler.
\end{proof}

\begin{theorem}[Isomorphism of cohomology on open strata]\label{thm:openCT}
    Suppose that $\varX \subseteq \varT$ is schön and cohomologically tropical. Let $\Sigma$ be any unimodular fan with support $\tropX =\trop(\varX)$. Then we obtain isomorphisms
     \[ \ct\colon H^k (\tropX^{\sigma}) \simto H^k(\varX^{\sigma}) \]
   for all $\sigma \in \Sigma$ and all $k$.
\end{theorem}

\begin{proof}
    It suffices to prove the statement for $\varX$, since if $\varX$ is cohomologically tropical so are all strata $\varX^{\sigma}$. It follows from the proof of \cref{thm:main}, that if $\varX$ is cohomologically tropical, then $\varX$ is wunderschön and hence the $2k$-th row of the $1$st page of the Deligne weight spectral sequence provides a resolution of $H^k(\varX)$ for all $k$. Moreover, the maps $\ctind{} \colon H^\bullet (\tropXbar^\sigma) \to H^\bullet (\varXbar^\sigma)$ are isomorphisms for all strata and they commute with the tropical and complex Gysin maps. Therefore, we obtain an isomorphisms of the resolutions which induces isomorphisms $\ctind{} \colon H^k(\tropX^\sigma) \to H^k (\varX^\sigma)$ for all $\sigma \in \Sigma$ and all $k$.
\end{proof}

 \section{Globalization} \label{sec:globalization}

We discuss a natural extension of the main theorem of \cite{IKMZ}. We follow the setting of that work. Let $\pi \colon \famX \to D^*$ be an algebraic family of non-singular complex algebraic varieties in $\CP^n$ over the punctured disk $D^*$.  By base change, the family $\famX$ gives rise to a subvariety $\famX_\eta$ of $\P^n_{\mathbb K}$ where $\mathbb K = \C((t))$ is the field of formal Laurent series with complex coefficients. (Here, $\eta$ refers to the generic fiber.) We denote by $Z\subseteq \TP^n$ the tropicalization of the family, defined as the extended tropicalization of the subvariety  $\famX_\eta \hookrightarrow \P^n_{\mathbb K}$.

We suppose $Z$ admits a unimodular triangulation. This is always possible after a base change of the form $D^* \to D^*$, $z\mapsto z^k$, for $k\in \Z_+$. Using the triangulation, we construct a degeneration of $\CP^n$ to an arrangement of toric varieties, and taking the closure of the family $\famX$ inside this toric degeneration leads to a family $\overline{\famX}$ extended over the full punctured disk $D$. By Mumford's proof of the semistable reduction theorem, we can always find a triangulation, after a suitable base change, such that the extended family is regular and the fiber over zero is reduced and simple normal crossing. This is known as a semistable extension of the family $\pi \colon \famX \to D^*$.

Denote by $\famX_0$ the fiber at zero of the extended family. Note that since the extended family is obtained by taking the closure of the family in a toric degeneration of $\CP^n$, each open stratum in $\famX_0$ will be naturally embedded in an algebraic torus. For $ t \in D^*$ denote by $\famX_t$ the fiber of $\pi$ over $t$.

\begin{theorem}\label{thm:globalization}
    Let $\pi \colon \famX \to D^*$ be an algebraic family of subvarieties in $\CP^n$ parameterized by the punctured disk and let $\pi \colon \overline \famX \to D$ be a semistable extension. If the tropicalization $Z\subseteq \TP^n$ is a tropical homology manifold and all the open strata in $\famX_0$ are wunderschön, then $H^{p,q}(Z)$ is isomorphic to the associated graded piece $W_{2p}/W_{2p-1}$ of the weight filtration in the limiting mixed Hodge structure $H_{\lim}^{p+q}$. The odd weight graded pieces in $H_{\lim}^{p+q}$ are all vanishing.

    Moreover, for $t \in D^*$, we have $\dim H^{p,q}(\famX_t) = \dim H^{p,q}(Z)$, for all non-negative integers $p$ and $q$.
\end{theorem}

\begin{proof}
    Denote by $\Delta$ the dual complex of $\famX_0$. By construction, $\Delta$ naturally lives in the tropicalization $Z$. For each simplex $\delta\in \Delta$, denote by $\varZ^\delta$ the corresponding (open) stratum in $\famX_0$ and denote by $Z^\delta$ the corresponding unimodular star fan in $Z$. Let $\comp{\varZ}^{\delta}$ and  $\comp{Z}^{\delta}$ be the corresponding compactifications. Note that $\comp{\varZ}^{\delta}$ is also the closure of $\varZ^\delta$ in $\famX_0$.

    Since $Z$ is a tropical homology manifold, the local fanfolds appearing in the tropical variety $Z$ are all tropical homology manifolds. Moreover, for each $\delta \in \Delta$, the Chow ring of the star fan $Z^\delta$ is also the tropical cohomology ring of $\comp{Z}^{\delta}$. Since $\varZ^\delta$ is cohomologically tropical, these rings are as well isomorphic to the cohomology ring of $\comp{\varZ}^\delta$.

    The Steenbrink spectral sequence~\cite{Ste76} for $\famX$ degenerates at page two and gives the weight $b$ part of the limit mixed Hodge structure $H^d_{\lim}$ in each cohomological degree $d$. The first page of the Steenbrink sequence is given by
    \[ \varST_1^{a,b} \coloneqq \bigoplus_{s\geq\abs a \\ s\equiv a\pmod 2} \varST_1^{a,b,s} \]
    where
    \[ \varST_1^{a,b,s}\coloneqq \bigoplus_{\delta\in \Delta \\ \dims\delta = s}H^{a+b-s}(\comp{\varZ}^{\delta}). \]
     The differential $\d \colon \varST_1^{a,b} \to \varST^{a+1, b}$ is given by a signed sum of Gysin and restriction maps.
     (The correct formulation involves taking the Tate twist $H^{a+b-s}(\comp{\varZ}^{\delta})(\frac{a-s}2)$ in $\varST^{a,b,s}$ so that all the terms have weight $b$; we omit them for sake of simplification.)
    
     Denoting the weight filtration in the limit mixed Hodge structure  $H^d_{\lim}$ by $W_{\bullet}H^d_{\lim}$, and setting  
    \[
    \grw_b H^{d}_{\lim} \coloneqq W_{b}H^d_{\lim}/W_{b-1}H^d_{\lim},
    \] 
    we have
    \[
        \grw_b H^{d}_{\lim} \simeq H^{d-b}(\varST^{\bullet,b}).
    \]
    The wunderschön assumption implies that the cohomology of $\varZ^\delta$ is of Hodge-Tate type for all $\delta\in \Delta$. By the Deligne weight spectral sequence, this implies that the same holds for the cohomology of $\comp{\varZ}^{\delta}$.  Therefore, the odd cohomology groups of $\comp{\varZ}^\delta$ all vanish. Since $a-s$ is even, we infer that for $b$ odd, all the terms $\varST_1^{a,b,s}$ vanish. In other words, the odd weight graded pieces of the limit mixed Hodge structures all vanish. Furthermore, the limit mixed Hodge structure will be of Hodge-Tate type.
    
    We set $b=2p$ and $d=p+q$, so that in weight $2p$ and degree $p+q$, we get
    \[
        \grw_{2p} H^{p+q}_{\lim} \simeq H^{q-p}(\varST^{\bullet,2p}).
    \]

    The tropical Steenbrink spectral sequence for the tropical homological manifold $Z$ is given in a similar way by 
    \[ \ST_1^{a,2p} \coloneqq \bigoplus_{s\geq\abs a \\ s\equiv a\pmod 2} \ST_1^{a,2p,s}, \qquad \ST_1^{a,2p,s} := \bigoplus_{\delta\in \Delta \\ \dims\delta = s}H^{a+2p-s}(\comp Z^{\delta}), \]
    with differential $\d \colon \ST_1^{a,2p} \to \ST^{a+1, 2p}$ a signed sum of Gysin and restriction maps. The Steenbrink-Tropical comparison theorem proved in \cite{AP-tht, IKMZ} implies that 
    \[
    H^{p,q}(Z) =  H^{q-p}(\ST^{\bullet,2p}).
    \]
    
    By assumption on the degeneration, we get isomorphisms between the cohomology groups $H^{\bullet}(\comp\varZ^{\delta}) \simeq H^{\bullet}(\comp Z^{\delta})$.  Comparing the two Steenbrink sequences $\varST^{\bullet, 2p}$ and $\ST^{\bullet, 2p}$, we deduce that $\grw_{2p} H^{p+q}_{\lim} \simeq H^{p,q}(Z)$. This gives the first part.
    
 To prove the second statement, note that since the limit mixed Hodge structure on $H^d_{\lim}$ is Hodge-Tate, the $p$-th part of the limit Hodge filtration $F^{\bullet}_{\lim}$ on $H^{d}_{\lim}$ contributes only in the graded piece $\grw_{2p} H^{p+q}_{\lim}$. This implies that the $H^{p,q}$ piece in the Hodge decomposition of $H^{p+q}(\famX_t)$ has dimension equal to that of $\grw_{2p}H^{p+q}(Z)$, finishing the proof of the theorem.
\end{proof}

From the proof, we deduce the following statement that shows that degenerations appearing in the above theorem are all maximal. 
\begin{theorem}\label{thm:max-degenerate} Notations as in Theorem~\ref{thm:globalization}, the family $\famX \to D^*$ is maximally degenerate. 
\end{theorem}
\begin{proof} Each closed stratum in $\famX_0$ has a cohomology of Hodge-Tate type. Steenbrink spectral sequence shows that the limit mixed Hodge structure is Hodge-Tate. 
\end{proof}

We discuss maximal degenerations further in Section~\ref{sec:discussion-max-degenerate}.

\section{Discussions}

\subsection{Examples}

In this section, we give various examples of varieties verifying some but not all conditions of the main \cref{thm:main}. These examples tend to demonstrate that the main theorem cannot be weakened.

\subsubsection{A wunderschön variety which is not cohomologically tropical}
    Take $N = \Z^2$. Let $\varX \subseteq \varT_N$ be the conic given by the equation $a + bz_1 + cz_2 + dz_1z_2 = 0$ for generic complex coefficients $a,b,c$ and $d$. The variety $\varX$ is $\CP^1$ with four points removed. This is a wunderschön variety: looking at the compactification $\varXbar \subseteq (\CP^1)^2$, $\varX$ is non-singular and the intersections with torus orbits are the points hence non-singular, so that $\varX$ is schön. Moreover, each of the points removed is trivially wunderschön. Finally, the Deligne weight spectral sequence shows that $\varX$ has pure Hodge structure. However, the tropicalization $\tropX$ of $\varX$ is the union of the axes in $\R^2$, which is not uniquely balanced, \ie, $\dim H^2(\tropX)=2>1$. This means that $\tropX$ is not a tropical homology manifold (see \cite[Theorem 4.8]{Aks21}). Moreover, computing the cohomology groups of $\tropXbar$, we obtain $\dim H^0(\tropXbar)=1$, $\dim H^1(\tropXbar) = 0$ and $\dim H^2(\tropXbar)=2$, which differs from the cohomology groups of the sphere $\varXbar$.

\smallskip
\subsubsection{A schön variety with pure strata, whose tropicalization is a tropical homology manifold but which is not cohomologically tropical}
    Let $\varX$ be a generic conic in $(\C^*)^2$. The variety $\varX$ is $\CP^1$ with six points removed. Its tropicalization is the usual tropical line equipped with weights equal to $2$ on all edges, hence again a tropical homology manifold by \cite[Theorem~4.8]{Aks21}. The variety $\varX$ is schön since it is non-singular, and each one of the three strata consists of two distinct points, hence it is non-singular. The mixed Hodge structure on $\varX$ is pure, as the Deligne weight spectral sequence shows that $\Gr_1^W H^1 (\varX)=H^1(\varXbar) = 0$. However, it is not wunderschön since its strata are not connected. The map $\ct\colon H^\bullet(\tropXbar) \to H^\bullet(\varXbar)$ is an isomorphism: it maps the class of a point in $\tropXbar$ to twice the class of a point in $\varXbar$. Nevertheless, $\varX$ is not cohomologically tropical since, for any ray $\zeta$ of $\tropX$, $H^0(\varXbar^\zeta)\cong \Q^2$ but $H^0(\tropXbar^\zeta)\cong \Q$.

\smallskip
\subsubsection{A schön variety which is not pure nor cohomologically tropical and whose tropicalization is a tropical homology manifold}
    Consider the punctured elliptic curve $\varX$ in $(\C^*)^2$ of equation $az_1^2+bz_2+cz_1z_2^2=0$ for generic complex coefficents $a,b$ and $c$. Topologically it is a torus punctured in three points. The tropicalization is the unimodular tropical line of weight one with rays generated by $(2,1),(-1,1)$ and $(-1,-2)$, which is a tropical homology manifold. The variety $\varX$ is non-singular and connected, and each of the three strata at infinity of its compactification is a point hence non-singular and connected. Hence $\varX$ is schön. The cohomology group $H^1(\varXbar)$ is nontrivial of dimension $2$. However, $H^1(\tropXbar)$ is trivial. Hence $\varX$ is not cohomologically tropical. This is because $\varX$ is not wunderschön. More precisely, $H^1(\varX)$ is not pure of weight~2. Indeed, by the Deligne weight spectral sequence $\Gr_1^W(H^1(\varX)) \cong H^1(\varXbar) \neq 0$.

\smallskip
\subsubsection{A non-schön variety which is cohomologically tropical}
    Once again, $N$ is of dimension 2. Let $\varX \subseteq \varT_N$ be given by the equation $(z_1-a)(z_2-b) = 0$ for $a,b \neq 0$. The variety $\varX$ is a reducible nodal curve with two components both being $\CP^1$ with two punctures. The tropicalization is again the union of the two coordinate axes in $\R^2$, which is not a tropical homology manifold, and the variety $\varX$ is not schön as it is singular. However, for each line of the cross, the cocycle associated to this line is mapped to the cocycle associated to the corresponding sphere. This is an isomorphism between $H^2(\tropX)$ and $H^2(\varX)$. Since $H^0(\tropX)$ is trivially isomorphic to $H^0(\varX)$ and other cohomology groups are trivial, we deduce that $\varX$ is cohomologically tropical.

\subsection{Hyperplane arrangement complements} \label{subsec:hyperplanecomps}

We will now see that all three properties of \cref{thm:main} are satisfied for complements of projective hyperplane arrangements. We will use the de Concini-Procesi model of the complement of a projective hyperplane arrangement~\cite{DP95}, as discussed in \cite[Section 4.1]{MS15}. Let $\arr=\set{H_i}_{i=0}^n$ be an arrangement of $n+1$ hyperplanes in $\P_\C^{d}$, not all having a common intersection point, and let $\varX_\arr =\P_\C^d \setminus \bigcup_{H_i \in \arr} H_i$ be the complement of the arrangement. For each $i$, let $\ell_i$ be the homogeneous linear form such that $H_i=\set{z\in \P_\C^d\mid \ell_i(z) = 0}$. These define a map $\varX_\arr \to \torus{n}$ given by $z \mapsto (\ell_i(z))$ in homogeneous coordinates on $\torus{n}$. This map is injective, since no $z\in \varX_\arr$ lies on all hyperplanes by assumption, and induces an isomorphism of $\varX_\arr \cong \varY_\arr$, where $\varY_\arr$ is a subvariety of $\torus{n}$, see \cite[Proposition 4.1.1]{MS15} for details. By a theorem of Ardila and Klivans~\cite{AK06}, the tropicalization $Y_\arr = \Trop(\varY_\arr)$ is the support of the Bergman fan $\Sigma_{M_\arr}$ of the matroid $M_\arr$ associated to the arrangement $\arr$, see \cite[Sections 4.1--4.2]{MS15}.

First, Tevelev shows \cite[Theorem 1.5]{Tev07} that the variety $\varX_\arr$ is schön, it is clearly connected, and by \cite{Sha93}, its cohomology has a pure Hodge structure of Hodge-Tate type. Moreover, given a face $\sigma \in \Sigma_{M_\arr}$, the star fan of $\sigma$ corresponds to the complement of a hyperplane arrangement. By induction, this shows that complements of hyperplane arrangements are wunderschön.

Furthermore, it is shown in \cite{JSS19,JRS18} by an inductive argument that the Bergman fan $\Sigma_{M_\arr}$ of a matroid is a tropical homology manifold. Therefore, one can apply \cref{thm:main}, which gives us that $\varY_\arr$ is cohomologically tropical, i.e. the map $\ctind{\varY_{\arr},\Sigma} \colon H^\bullet (\comp{Y}_\arr) \to H^{\bullet}(\comp\varY_\arr)$ is an isomorphism.

In light of \cref{thm:openCT}, this can be compared with the main result of \cite{Zha13}, also independently proved in \cite{Sha-thesis}, showing that $H^\bullet(\varX_\arr)\cong H^\bullet(\tropX_\arr)$, however lacking explicit maps.

\subsection{Non-matroidal examples} \label{subsec:non-matroidal_example}
We say that a subvariety $\varX$ of a torus $\varT_N$ is \emph{non-matroidal} if the tropicalization  of $\varX$ is not a Bergman fanfold, up to isomorphisms of the lattice $N$.

We present an example of $\varX \subseteq \varT_N$ which is not a complement of a hyperplane arrangement yet is wunderschön, cohomologically tropical, and the tropicalization $\trop(\varX)$ is a tropical homology manifold.

The variety $\varX$ will be the complement of an arrangement of lines and a single conic in $\CP^2$. Let $[z_0:z_1: z_2]$ be homogeneous coordinates on $\CP^2$. Let $\varL_0$, $\varL_1$, and $\varL_2$ be the coordinate lines of $\CP^2$ so that $\varL_i $ is defined by $z_i = 0$. Let $\varL_3$ be defined by the linear form $z_0 - z_1 + z_2 = 0$ and let the conic $\varC$ be defined by $z_1^2 +z_2^2 - z_0z_1 - 2z_1z_2 = 0$. Let $\arr$ denote the union of $\varL_0, \dots, \varL_3, \varC$.

As depicted in \cref{fig:non-matroidal_example_varX}, note that $\varC$ is tangent to $\varL_1$ at the point $[1:0:0]$ where $\varL_1$ intersects $\varL_2$. Also the conic $\varC$ is tangent to $\varL_0$ at the intersection point $[0:1:1]$ with $\varL_3$. The conic also passes through the intersection point $[1:1:0]$ of $\varL_2$ and $\varL_3$.

\begin{figure}
    \begin{tikzpicture}
        \input{figures/non_matroidal_example.tex}
    \end{tikzpicture}
    \caption{A very-affine variety $\varX$ which is not a complement of hyperplane arrangement and verifies the main theorem, see \cref{subsec:non-matroidal_example}. \label{fig:non-matroidal_example_varX}}
\end{figure}

Consider the map $\phi \colon \CP^2 \setminus \arr \to (\C^*)^4$ defined by
\[ [z_0: z_1:z_2] \mapsto (\~z_1, \~z_2, 1 - \~z_1 + \~z_2, \~z_1^2 +\~z_2^2 - \~z_1 - 2\~z_1\~z_2 ), \quad\text{ with $\~z_1 = \frac{z_i}{z_0}$ and $\~z_2 = \frac{z_2}{z_0}$.}\]
Let $\varX \subseteq (\C^*)^4$ denote the image of the map $\phi$. The space $\trop(\varX)$ is $2$-dimensional and is the support of the fan described below.

\begin{figure}
    \[ \begin{array}{cccccccc}
         0 & 1 & 2 & 3 & 4 & a & b &  c \\\hline
        -1 & 1 & 0 & 0 & 0 & 2 & 0 & -2 \\
        -1 & 0 & 1 & 0 & 0 & 1 & 1 & -2 \\
        -1 & 0 & 0 & 1 & 0 & 0 & 1 & -1 \\
        -2 & 0 & 0 & 0 & 1 & 2 & 1 & -2
    \end{array} \qquad \begin{array}{cc}
        \alpha & \beta \\\hline
        1 & -1 \\
        1 & -1 \\
        0 &  0 \\
        0 & -1
    \end{array} \qquad
    \begin{tikzpicture}[inner sep=1pt, baseline={(0,0)}]
        \node (0) at (-1,1) {0};
        \node (1) at (1,1) {1};
        \node (2) at (1,-1) {2};
        \node (3) at (-1,-1) {3};
        \node (4) at (0,.5) {4};
        \node (a) at (1,.5) {$a$};
        \node[gray] (alp) at (1,0) {$\alpha$};
        \node (b) at (0,-1) {$b$};
        \node (c) at (-1,.5) {$c$};
        \node[gray] (bet) at (-1,0) {$\beta$};
        \foreach \i/\j in {0/1,0/2,0/c,c/bet,bet/3,1/3,1/a,a/alp,alp/2,2/b,3/b} {
            \draw (\i) -- (\j);
        }
        \draw (4) edge (a) edge[bend left=15] (b) edge (c);
    \end{tikzpicture} \]
    \caption{The combinatorial structure of a non-Bergman fan verifying the main theorem described in \cref{subsec:non-matroidal_example}. \label{fig:non-matroidal_example_tropX}}
\end{figure}

The fan has $8$ rays in directions given in \cref{fig:non-matroidal_example_tropX}. Each ray is adjacent to exactly $3$ faces of dimension $2$ for a total of $12$ faces of dimension $2$. The structure is given in \cref{fig:non-matroidal_example_tropX}: we draw an edge between two vertices if the there is a face between the two corresponding rays. Note that to get a unimodular subdivision, one has to add some rays, for instance the rays $\alpha$ and $\beta$ of \cref{fig:non-matroidal_example_tropX}. We denote by $\Sigma$ this unimodular fan.

It can be verified in polymake that this fan is a tropical homology manifold and its tropical Betti numbers are $1,0,6,0,1$. For an alternative proof, note that the the fan $\Sigma$ is obtained by the process of tropical modification \cite{MR18} as follows. Let $\Sigma_{U_{3,4}} \subseteq \R^3$ be the Bergman fan of the uniform matroid $U_{3,4}$. Its rays are the rays $0,1,2$ and $3$ in \cref{fig:non-matroidal_example_tropX}, where we forget the fourth coordinate. Let $C \subseteq \Sigma_{U_{3,4}}$ be a tropical trivalent curve with rays $a,b,c$ (once again we forget the last coordinate). Then $\Sigma$ in $\R^4$ is obtained by a tropical modification of $\Sigma_{U_{3,4}}$ along $C$. By \cite{JRS18}, the Bergman fan $\Sigma_{U_{3,4}}$ is a tropical homology manifold, see also \cref{subsec:hyperplanecomps}. By \cite[Theorem 4.8]{Aks21} the trivalent tropical curve is also a tropical homology manifold. By \cite[Theorem 1.4]{AP-homology} the modification of $\Sigma_{U_{3, 4}} \subseteq \R^3$ along $C$ is a tropical homology manifold. The tools developed in this last article also allow to compute the cohomology of $\tropXbar$ quite easily, and to check that the fan is Kähler.

The compactification of $\varX$ in $\CP_{\Sigma}$ is given as follows. Consider $\CP^2$ blown up in the three points whose homogeneous coordinates are $[1:0:0], [0:1:1],$ and $[1:1:0]$. Then, in the blow up, the exceptional divisor above $[1:0:0]$, the proper transform of $\varC$, and the proper transform of $\varL_1$ all intersect in a single point. Similarly, there is a triple intersection of the exceptional divisor above $[0:1:1]$ and the proper transforms of $\varC$ and $\varL_0$. We further blow-up these two intersection points to obtain a surface $\varXbar$. The divisor $\varXbar \setminus \varX$ consists of the five exceptional divisors and the proper transforms of all curves in $\arr$. Therefore, $\dim H^2(\varXbar) = 6$ and $\dim H^0(\varXbar) = \dim H^4(\varXbar) = 1$ and $\dim H^k(\varXbar) = 0$ otherwise.

We claim that $\varX$ is wunderschön. Indeed, for each ray $\zeta$ of the fan $\Sigma$ the variety $\varXbar^\zeta$ is $\CP^1$ with two or three marked points corresponding to the intersections with the other divisors in $\varXbar \setminus \varX$, so it is wunderschön. Moreover, $\varX$ is non-singular and connected, and its cohomology is pure. Hence, $\varX$ is wunderschön.

In \cite{Aks24}, tropicalizations of complements of curve arrangements are considered more generally from the perspective of determining when they are cohomologically tropical. In particular, \cite[Example 7.4]{Aks24} gives an infinite family of non-matroidal cohomologically tropical varieties of dimension two.

 Other examples of cohomologically tropical varieties include the moduli spaces $M(3,6)$ and $M(3,7)$ of 6 and 7 lines in $\CP^2$, respectively, see~\cite{Lu08, Sch21, CL23}. 

\smallskip

By the tropical and classical Künneth formula, the product of two cohomologically tropical varieties is again cohomologically tropical. Moreover, in a product $\varX_1 \times \varX_2 \subset \varT_{N_1} \times \varT_{N_2}$, if one of the factors is non-matroidal, the product remains non-matroidal (because the same statement holds for a product of fanfolds, using the fact that local fanfolds of a Bergman fanfold are all Bergman).

 Using the family of non-matroidal examples above, as well as the complements of hyperplane arrangements in any dimension, we obtain infinite families of non-matroidal cohomologically tropical varieties in arbitrary dimension at least two.

\subsection{Maximal degenerations} \label{sec:discussion-max-degenerate} Motivated by the work of Deligne~\cite{Del-md} and our results Theorem~\ref{thm:globalization} and Theorem~\ref{thm:max-degenerate}, we can ask the following question.
\begin{question} Is there a tropical geometric characterization of maximally degenerate families of complex algebraic varieties?
Is it true that those families in which the open strata of special fibers have a cohomology which is pure of Hodge-Tate type are exactly those covered by our Theorem~\ref{thm:globalization}? 
\end{question}

These questions are intimately related to open problems on large scale limits of families of complex algebraic varieties. For instance, a recent work by Yang Li~\cite{Li20} reduces the metric SYZ conjecture in maximally degenerate families of complex algebraic varieties to the existence of solutions to a tropical Monge-Ampère equation, once this equiation has been properly formulated. For those degenerations appearing in Theorem~\ref{thm:globalization}, our results show that the corresponding tropical variety is Kähler in the sense of~\cite{AP-tht} and moreover recovers the geometry of the degenerate fiber as well as the limit Hodge-theoretic geometry of the family.  In a forthcoming work~\cite{AP-MA}, a differential calculus on tropical varieties is developed  that combines Chow rings of local tropical fans with real differential forms on the variety. Tropical Hodge theory can be then used to properly formulate the Monge-Ampère equation on the tropicalization using tropical Kähler forms. This allows to formulate a tropical analogue of the Monge-Ampère equation, and study its solutions. 
 
 Another instance of large scale asymptotic problem is a conjecture by Kontsevich and Soibelman~\cite{KS06} that predicts the convergence of the normalized Calabi metrics in maximally degenerate families of Calabi-Yau varieties to a tropical Calabi metric (once this has been properly defined) on the limit tropical variety. 

\subsection{$\mT$-stability} \label{sec:shellability}

It seems plausible that a framework parallel to the one in~\cite{AP-hodge-fan} can be developed for properties of tropicalization of algebraic varieties. The properties discussed in this paper concern pairs consisting of a subvariety of an algebraic torus and a fan structure on its tropicalization. Three basic operations can be conducted on these pairs: products, blow-ups and blow-downs, and taking the graph of a holomorphic function on the subvariety, restricted to the complement of its divisor. For example the cases described in \cref{subsec:hyperplanecomps,subsec:non-matroidal_example} can both be obtained by these operations. The properties of being schön, wunderschön, and cohomologically tropical should be $\mT$-stable in this framework. We refer to~\cite{Sch21} for some results in this direction.

\printbibliography

\end{document}

%% file: figures/non_matroidal_example.tex

\begin{scope}[scale=1.2]
  \draw (0,-1) -- (0,1.5) node[above] {$\varL_1$};
  \draw (-.5,0) -- (1.5,0) node[right] {$\varL_2$};
  \draw (-.5,1.707) -- (1.707,-.5) node[below] {$\varL_0$};
  \draw (-.5,-1) -- (1.3,.8) node[right] {$\varL_3$};
  \draw (0,0) to[out=90, in=135, looseness=1.1] node[midway, below right=-2pt] {$\varC$} (1.707/2,.707/2) to[out=-45,in=22.5, looseness=1] (.5,0) to[out=-180+22.5,in=-90, looseness=1] cycle;
\end{scope}